\RequirePackage{fix-cm}
\documentclass[smallextended]{svjour3}       
\smartqed  

\smartqed 

\usepackage{amsmath}
\usepackage{amssymb}
\usepackage{amsfonts}
\usepackage{hyphenat}
\usepackage{hyperref}
\usepackage{times}
\usepackage{psfrag}
\usepackage{ifpdf}
\usepackage{array}
\usepackage{multirow}
\usepackage{setspace}
\usepackage{latexsym}
\usepackage{enumerate}
\usepackage{mathtools}
\usepackage{float}
\usepackage{scalerel}
\usepackage{cite}
\usepackage{graphicx}
\usepackage{caption}
\usepackage{epstopdf}
\usepackage[table]{xcolor}
\usepackage{subfig}
\usepackage{accents}
\usepackage{algorithm}
\usepackage{algpseudocode}
\usepackage{mathrsfs}
\usepackage{tabularx}
\usepackage{tikz}
\usepackage{pgfplots}
\usepackage{url}  
\usepackage{caption}
\usepackage{rotating}
\usepackage{adjustbox}

\usepackage[numbers,sort&compress]{natbib}

\newcommand{\sm}[2]{\scaleto{#1\mathstrut}{#2pt}}

\newcommand{\msh}{\!\!\:}

\definecolor{Fcolor}{rgb}{0,0.5,0.25}

\newcommand{\ignore}[1]{}

\newcommand{\tr}{\mathrm{tr}}

\newcommand{\T}{^{\!\top}}
\newcommand{\Tp}{^{\!\phantom{\top}}}

\newcommand{\la}{\langle}
\newcommand{\ra}{\rangle}

\newcommand{\abf}{\boldsymbol{a}}

\newcommand{\Abf}{\boldsymbol{A}}

\newcommand{\bbf}{\boldsymbol{b}}

\newcommand{\Bbf}{\boldsymbol{B}}

\newcommand{\Dbb}{\mathbb{D}}

\newcommand{\ebf}{\boldsymbol{e}}

\newcommand{\Ecal}{\mathcal{E}}

\newcommand{\Fbf}{\boldsymbol{F}}
\newcommand{\Fcal}{\mathcal{F}}
\newcommand{\Frm}{\mathrm{F}}

\newcommand{\Gbf}{\boldsymbol{G}}
\newcommand{\Gcal}{\mathcal{G}}

\newcommand{\Hbf}{\boldsymbol{H}}

\newcommand{\Ibf}{\boldsymbol{I}}

\newcommand{\Ical}{\mathcal{I}}

\newcommand{\Ncal}{\mathcal{N}}

\newcommand{\Ocal}{\mathcal{O}}
\newcommand{\Ob}{\boldsymbol{O}}
\newcommand{\obf}{\boldsymbol{o}}

\newcommand{\Pcal}{\mathcal{P}}

\newcommand{\Rbb}{\mathbb{R}}
\newcommand{\Rcal}{\mathcal{R}}

\newcommand{\Sbb}{\mathbb{S}}
\newcommand{\Scal}{\mathcal{S}}

\newcommand{\Tbf}{\boldsymbol{T}}

\newcommand{\Ubf}{\boldsymbol{U}}

\newcommand{\Ucal}{\mathcal{U}}

\newcommand{\Vcal}{\mathcal{V}}

\newcommand{\wbf}{\boldsymbol{w}}

\newcommand{\wht}{\hat{\boldsymbol{w}}}

\newcommand{\xbf}{\boldsymbol{x}}
\newcommand{\Xbf}{\boldsymbol{X}}

\newcommand{\xst}{\accentset{\ast}{\xbf}}
\newcommand{\Xst}{\accentset{\ast}{\Xbf}}

\newcommand{\ybf}{\boldsymbol{y}}
\newcommand{\Ybf}{\boldsymbol{Y}}
\newcommand{\Ych}{\check{\Ybf}}
\newcommand{\Yst}{\accentset{\ast}{\Ybf}}

\newcommand{\Zbf}{\boldsymbol{Z}}

\newcommand{\nubf}{\boldsymbol{\nu}}
\newcommand{\taubf}{\boldsymbol{\tau}}

\newcommand{\nuht}{\hat{\nu}}

\newcommand{\Lambdabf}{\boldsymbol{\Lambda}}

\makeatletter
\DeclareRobustCommand{\cev}[1]{%
  \mathpalette\do@cev{#1}%
}
\newcommand{\do@cev}[2]{%
  \fix@cev{#1}{+}%
  \reflectbox{$\m@th#1\vec{\reflectbox{$\fix@cev{#1}{-}\m@th#1#2\fix@cev{#1}{+}$}}$}%
  \fix@cev{#1}{-}%
}
\newcommand{\fix@cev}[2]{%
  \ifx#1\displaystyle
    \mkern#2 1mu
  \else
    \ifx#1\textstyle
      \mkern#2 3mu
    \else
      \ifx#1\scriptstyle
        \mkern#2 2mu
      \else
        \mkern#2 2mu
      \fi
    \fi
  \fi
}

\begin{document}

\title{Parabolic Relaxation for Quadratically-constrained Quadratic Programming -- Part I: 
Definitions \& Basic Properties
	\thanks{This work is in part supported by the NSF Award 1809454. Alper Atamt\"{u}rk is supported, in part, by grant FA9550-10-1-0168 from the Office of the Assistant Secretary of Defense for Research \& Engineering and the NSF Award 1807260.}
}


\author{Ramtin Madani \and Mersedeh Ashraphijuo \and Mohsen Kheirandishfard \and Alper Atamt\"{u}rk}

\institute{Ramtin Madani \and Mersedeh Ashraphijuo \at
	Department of Electrical Engineering, University of Texas, Arlington\\
	\email{ramtin.madani@uta.edu, mersedeh.ashraphijuo@uta.edu}           
	\and
	Mohsen Kheirandishfard \at
	Cognitiv \\
	\email{mohsen.kheirandishfard@gmail.com}
	\and
	Alper Atamt\"{u}rk \at
	Department of Industrial Engineering \& Operations Research, University of California, Berkeley \\
	\email{atamturk@berkeley.edu}   
}

\date{Received: date / Accepted: date}

\maketitle


\begin{abstract}
For general quadratically-constrained quadratic programming (QCQP), we propose a {\it parabolic relaxation} described with convex quadratic constraints. An interesting property of the parabolic relaxation is that the original non-convex feasible set is contained on the boundary of the parabolic relaxation. Under certain assumptions, this property enables one to recover near-optimal feasible points via objective penalization. Moreover, through an appropriate change of coordinates that requires a one-time computation of an optimal basis, the easier-to-solve parabolic relaxation can be made as strong as a semidefinite programming (SDP) relaxation, which can be effective in accelerating algorithms that require solving a sequence of convex surrogates. The majority of theoretical and computational results are given in the next part of this work \cite{parabolic_part2}. 
	
\end{abstract}

\section{Introduction}

In this paper, we study the minimization of a quadratic objective function over a feasible set that is also defined by quadratic functions. This problem is referred to as quadratically-constrained quadratic programming (QCQP). Physical laws and characteristics of dynamical systems are widely modeled using quadratic functions. Additionally, the general class of polynomial optimization problems can be reduced to QCQP. Consequently, QCQP arises in a wide range of scientific and engineering applications, such as power systems 
\cite{madani2016promises,madani2015convex,madani2015convexification,madani2014low}, 
imaging science 
\cite{bandeira2014tightness,candes2015phase,fogel2013phase,singer2011angular}, 
signal processing \cite{aittomaki2009beampattern,aubry2013ambiguity,chen2009mimo,li2012approximation,luo2010semidefinite,mariere2003blind}, automatic control
\cite{fazelnia2014convex,toker1998complexity,madani2014low,ahmadi2017dsos}, 
quantum mechanics 
\cite{hilling2010geometric,deza1994applications,laurent2015conic,burgdorf2015closure},
and cybersecurity
\cite{cid2005small,cid2004computational,courtois2002cryptanalysis,murphy2002essential}. 
The development of efficient optimization techniques and numerical algorithms for QCQP has been an active area of research for decades. 
Although QCQP is NP-hard, specialized methods that are robust and scalable are available for specific application domains.
Notable examples for which methods with guaranteed performance has been offered in the literature include problems of multi-sensor beam-forming in communication theory \cite{gershman2010convex}, phase retrieval in signal processing \cite{candes2013phaselift}, and matrix completion in machine learning \cite{mu2016scalable,candes2009exact}.

This paper advances convex relaxation framework for solving general QCQP problems  \cite{sherali1990hierarchy,nesterov1994interior,lasserre2001global,lasserre2006convergent,josz2015moment,chen2012globally,papp2013semidefinite,mohammad2017rounding}. A relaxation is said to be \textit{exact} if it has the same optimal objective value as the original problem. Convex relaxations have been critically important for tackling non-convex optimization problems and their effectiveness has been verified for numerous real-world problems \cite{lasserre2001explicit,kim2003exact,sojoudi2013exactness1,sojoudi2013exactness2,burer2018exact}. 
For challenging cases of QCQPs, convex relaxations offer effective approximation algorithms \cite{Nesterov98,Ye991,Ye992,Zhang00,Zhang06,Luo07,He08,He10}. 
For example, Geomans and Williamson \cite{goemans1995improved} show that the objective of a semidefinite programming (SDP) relaxation is within 14\% of the optimal value for the MAXCUT problem on graphs with non-negative weights.
Additionally, convex relaxations can be integrated within branch-and-bound algorithms \cite{chen2017spatial,burer2008finite} for finding near-optimal solutions to non-convex optimization problems. Despite solid theoretical guarantees, one of the primary challenges for the application of existing convex relaxations beyond small-scale instances is the rapid growth of dimensionality. In \cite{nohra2021sdp}, for a single quadratic constraint and a set of linear equations, the authors give a convex  semi-infinite program with non-convex quadratic inequalities that achieves the standard SDP relaxation bound.

General purpose convex relaxations such as SDP have a wide variety of applications in many areas including machine learning \cite{bennett2007netflix,d2004direct,kulis2007fast}, control theory \cite{andersen2010implementation,andersen2014robust,benson2006dsdp5}, and robotics \cite{mangelson2019guaranteed,thrun2005probabilistic}. Historically, the lack of scalability has been a major impediment to fulfilling the full potential of SDP for these application domains \cite{majumdar2019recent}. To address scalability challenges, a number of approaches are pursued in the literature. Several papers focus on leveraging structures such as symmetry and chordal sparsity \cite{lasserre2017bounded,weisser2018sparse}. Another direction relies on the existence of low-rank solutions for SDP \cite{AG:rank-one,Barvinok95,lemon2016low, MAD:scalable,Pataki98}. Numerical algorithms such as the alternating direction method of multipliers (ADMM) are used to tackle the scalability issues of second-order numerical methods as well \cite{yang2015sdpnal,zhao2010newton,o2016conic}. 

In this work, we offer an alternative convex relaxation for QCQP, namely {\it parabolic relaxation}, which relies solely on convex quadratic inequalities, as opposed to SDP constraints which can be computationally demanding \cite{majumdar2019recent}. An interesting property of the parabolic relaxation is that the original non-convex feasible set is contained on the boundary of the parabolic relaxation, which we refer to as the \textit{comprehensive boundary property} (CBP).
Thanks to CBP,
we prove that, under certain conditions, incorporating an appropriate linear penalty term in the objective of the parabolic relaxation can remedy inexactness and lead to recovering high-quality feasible points for non-convex QCQPs.
This penalty term can be constructed based on an point which is sufficiently close to the feasible set and satisfies a \textit{generalized linear independence constraint qualification} (GLICQ) condition. 
In place of computationally demanding relaxations such as SDP, we adopt a simple and tractable parabolic relaxation and a sequential penalty approach to find near-optimal feasible points.

In this part, we introduce the parabolic relaxation and prove that it satisfies the comprehensive boundary property for QCQP.
This is followed by a sequential penalized parabolic relaxation algorithm, and comparisons of the parapolic relaxation with other approaches, such as SDP relaxation \cite{nesterov1994interior}, SOCP relaxation \cite{alizadeh2003second}, diagonally-dominant (DD) and scaled diagonally-dominant (SDD) approximations \cite{ahmadi2017dsos} are provided. We also show that the lower bound obtained from the parabolic relaxation can be made as strong as the SDP relaxation bound through an appropriate basis change. This result allows one to replace an SDP relaxation with a cheaper-to-compute parabolic relaxation in algorithms that that require solving a sequence of convex surrogates. Lastly, we provide examples comparing the relaxations and demonstrating the sequential algorithm. 
The majority of theoretical and computational results are given in the second part of this work \cite{parabolic_part2}. 

\subsection{Notations}

For a given vector $\abf$ and a matrix $\Abf$, the symbols $a_i$ and $A_{ij}$, respectively, indicate the $i^{th}$ element of $\abf$ and the $(i,j)^{\mathrm{th}}$ element of $\Abf$.
The symbols $\Rbb$, $\Rbb^{n}$, and $\Rbb^{n\times m}$ denote the sets of real scalars, real vectors of size $n$, and real matrices of size $n\times m$, respectively. The set of $n\times n$ real symmetric matrices is shown by $\Sbb_n$. 
The notations $\Dbb_n$ and $\Sbb_n^{+}$, respectively, represent the sets of diagonally-dominant and positive semidefinite members of $\Sbb_n$. The $n\times n$ identity matrix is denoted by $\Ibf_n$. The origins of $\Rbb^n$ and $\Rbb^{n\times m}$ are denoted by $\boldsymbol{0}_n$ and $\boldsymbol{0}_{n\times m}$, respectively, while $\boldsymbol{1}_n$ denotes the all one $n$-dimensional vector. For a pair of $n\times n$ symmetric matrices $(\Abf,\Bbf)$, the notation $\Abf\succeq\Bbf$ means that $\Abf-\Bbf$ is positive semidefinite, whereas $\Abf\succ\Bbf$ means that $\Abf-\Bbf$ is positive definite. Given a matrix $\Abf\in\mathbb{R}^{m\times n}$, the notation $\sigma_{\min}(\Abf)$ represents the minimum singular value of $\Abf$. 
The symbols $\langle\cdot\,,\cdot\rangle$ and $\|\cdot\|_{\Frm}$ denote the Frobenius inner product and norm of matrices, respectively. The notations $\|\cdot\|_1$ and $\|\cdot\|_2$ denote the $1$-  and $2$-norms of vectors/matrices, respectively. The superscript $(\cdot)^\top$ and the symbol $\mathrm{tr}\{\cdot\}$ represent the transpose and trace operators, respectively. $\mathrm{vec}\{\cdot\}$ denotes the vectorization operator. The notation $|\cdot|$ represents either the absolute value operator or cardinality of a set, depending on the context. For every $\xbf\in\Rbb^n$, the notation $[\xbf]$ represents an $n\times n$ diagonal matrix with the elements of $\xbf$. The operators $\mathrm{conv}\{\cdot\}$ and $\partial$ represent convex hull and boundary of sets, respectively. The boundary of a subset $S$ of a topological space $X$ is defined as the set of points in the closure of $S$ not belonging to the interior of $S$. Given a pair of integers $(n,r)$, the binomial coefficient ``$n$ choose $r$'' is denoted by {\footnotesize $\Big(\begin{array}{c} n \\ r \end{array}\Big)$}. 

\section{Problem Formulation}

In this section, we introduce the parabolic relaxation for quadratically-constrained quadratic programming and its tensor generalization.

\subsection{Parabolic Relaxation for Quadratically-constrained Quadratic Programming}

Consider the class of general quadratically-constrained quadratic programming (QCQP)
\begin{subequations}
	\begin{align}
		& \underset{\xbf\in\Rbb^n}{\text{minimize}}
		& & q_0(\xbf)  & \label{QCQP_vector_a}\\
		& \text{subject to}
		& &   q_k(\xbf)\leq0,   &&&& k\in\Ical   \label{QCQP_vector_b}, \\
		& & & q_k(\xbf)=0,      &&&& k\in\Ecal   \label{QCQP_vector_c},
	\end{align}
\end{subequations}	
where $\Ical$ and $\Ecal$ denote the sets of inequality and equality constraints, respectively. For each $k\in\{0\}\cup\Ical\cup\Ecal$, define $q_k:\Rbb^n\to\Rbb$ as a possibly non-convex quadratic function of the form $q_k(\xbf)\triangleq\xbf\T\!\Abf_k\xbf+2\bbf\T_k\xbf+c_k\Tp$, where the matrix $\Abf_k\in\Sbb_n$, the vector $\bbf_k\in\Rbb^n$, and the scalar $c_k\in\Rbb$ are given. 

In order to formulate convex relaxations for  problem \eqref{QCQP_vector_a}--\eqref{QCQP_vector_c}, we employ the well-known {\it lifting} technique, where an auxiliary matrix variable $\Xbf\in\Sbb_n$ is introduced as a proxy for $\xbf\xbf^{\!\top}$ and problem \eqref{QCQP_vector_a}--\eqref{QCQP_vector_c} is reformulated as:
\begin{subequations}
	\begin{align}
		& \underset{\begin{subarray}{l}\phantom{\Xbf}\xbf\in\Rbb^n\\ \phantom{\xbf}\Xbf\in\Sbb_n\end{subarray}}{\text{minimize}}
		& & \bar{q}_0(\xbf, \Xbf)  & \label{QCQP_vector_lifted_a}\\
		& \text{subject to}
		& & \bar{q}_k(\xbf, \Xbf)\leq 0, &&&& k\in\Ical,   \label{QCQP_vector_lifted_b} \\
		& & & \bar{q}_k(\xbf, \Xbf)=0,  &&&& k\in\Ecal,   \label{QCQP_vector_lifted_c} \\
		& & & \Xbf=\xbf\xbf\T,  &&&&  \label{QCQP_vector_lifted_d}
	\end{align}
\end{subequations}	
where, for each $k\in\{0\}\cup\Ical\cup\Ecal$, the linear function $\bar{q}_k:\Rbb^n\times\Sbb_n\to\Rbb$ is defined as $\bar{q}_k(\xbf,\Xbf)\triangleq\mathrm{tr}\{\Abf_k\T\Xbf\}+2\bbf^{\!\top}_k\xbf+c_k$.  Define $\Ncal\triangleq\{1,\ldots,n\}$.
 The most common method to tackle the non-convex constraint \eqref{QCQP_vector_lifted_d} is to relax it to
\begin{align}
\Xbf\succeq\xbf\xbf\T,
\end{align}
which is referred to as the semidefinite programming (SDP) relaxation \cite{nesterov1994interior}. A more tractable alternative is
\begin{align} \label{common-socp}
\begin{bmatrix}
 X_{ii} & X_{ij}\\
 X_{ij} & X_{jj}
\end{bmatrix}\succeq 0\quad\forall i,j\in\Ncal,\qquad\qquad
\begin{bmatrix}
 X_{ii} & x_{i}\\
 x_{i} & 1
\end{bmatrix}\succeq 0\quad\forall i\in\Ncal,
\end{align}
which is referred to as the second-order cone programming (SOCP) relaxation \cite{alizadeh2003second}.

\pagebreak

Here we introduce the parabolic relaxation as an alternative to the SDP and SOCP relaxations described above. Let $\{\wbf_i\in\Rbb^n\}_{i\in\Ncal}$ represent an arbitrary basis for $\Rbb^n$. Observe that, the constraint \eqref{QCQP_vector_lifted_d} can be reformulated as:
\begin{align}
&w_i\T \Xbf w_j= w_i\T\xbf (w_j\T\xbf)\T, && \forall i,j\in\Ncal,
\end{align}
which is equivalent to:
\begin{subequations}
\begin{align}
& (\wbf_i+\wbf_j)^{\!\top}\Xbf(\wbf_i+\wbf_j)= 
\big((\wbf_i+\wbf_j)\T\xbf\big)^2,  && i,j\in\Ncal, \label{rel1}\\
& (\wbf_i-\wbf_j)^{\!\top}\Xbf(\wbf_i-\wbf_j)= 
\big((\wbf_i-\wbf_j)\T\xbf\big)^2,  && i,j\in\Ncal, \label{rel2}
\end{align}
\end{subequations}
By relaxing \eqref{rel1} -- \eqref{rel2} to convex inequalities, we arrive at the \textit{parabolic relaxation} of QCQP:
\begin{subequations}
	\begin{align}
		& \underset{\begin{subarray}{l}\phantom{\Xbf}\xbf\in\Rbb^n\\ \phantom{\xbf}\Xbf\in\Sbb_n\end{subarray}}{\text{minimize}}
		& & \bar{q}_0(\xbf, \Xbf)  & \label{QCQP_vector_para_a}\\
		& \text{subject to}
		& & \bar{q}_k(\xbf, \Xbf)\leq 0, &&&& k\in\Ical,   \label{QCQP_vector_para_b} \\
		& & & \bar{q}_k(\xbf, \Xbf)=0,  &&&& k\in\Ecal,   \label{QCQP_vector_para_c} \\
		& & & (\wbf_i+\wbf_j)^{\!\top}\Xbf(\wbf_i+\wbf_j)\geq 
		\big((\wbf_i+\wbf_j)\T\xbf\big)^2,  &&&& i,j\in\Ncal, \label{QCQP_vector_para_d}\\
		& & & (\wbf_i-\wbf_j)^{\!\top}\Xbf(\wbf_i-\wbf_j)\geq 
		\big((\wbf_i-\wbf_j)\T\xbf\big)^2,  &&&& i,j\in\Ncal, \label{QCQP_vector_para_e}
	\end{align}
\end{subequations}	
in which the convex quadratic constraints \eqref{QCQP_vector_para_d} -- \eqref{QCQP_vector_para_e} are proxies for the non-convex constraint $\Xbf=\xbf\xbf\T$.


The term \textit{parabolic} is motivated by the fact that  \eqref{QCQP_vector_para_d} -- \eqref{QCQP_vector_para_e} can be cast as
\begin{subequations}
\begin{align}
&\big((\wbf_i+\wbf_j)\T\xbf,\;  (\wbf_i+\wbf_j)^{\!\top}\Xbf(\wbf_i+\wbf_j)\big)\in\mathbb{P},  && i,j\in\Ncal, \\
&\big((\wbf_i-\wbf_j)\T\xbf,\;  (\wbf_i-\wbf_j)^{\!\top}\Xbf(\wbf_i-\wbf_j)\big)\in\mathbb{P},  && i,j\in\Ncal, 
\end{align}
\end{subequations}
where $\mathbb{P}\triangleq\{(\mathsf{x},\mathsf{y})\in\Rbb^2\,|\,\mathsf{y}\geq\mathsf{x}^2\}$ represents the epigraph of a two-dimensional parabola. For the standard basis for $\Rbb^n$, 
parabolic constraints \eqref{QCQP_vector_para_d} -- \eqref{QCQP_vector_para_e} reduce to: 
\begin{subequations} \label{eq:standard-parabolic}
	\begin{align}
		X_{ii}+X_{jj}+2X_{ij}\geq (x_i+x_j)^2,\quad\quad\forall i,j\in\Ncal, \label{QCQP_vector_para_stan_d}\\
		X_{ii}+X_{jj}-2X_{ij}\geq (x_i-x_j)^2,\quad\quad\forall i,j\in\Ncal. \label{QCQP_vector_para_stan_e}	
	\end{align}
\end{subequations}
Observe that the parabolic inequalities \eqref{QCQP_vector_para_d} -- \eqref{QCQP_vector_para_e} can also be written as second order cone constraints:
\begin{subequations}
\begin{align}
&(\wbf_i+\wbf_j)^{\!\top}\Xbf(\wbf_i+\wbf_j)+1\geq \nonumber\\
&\qquad\qquad\qquad\qquad\left\|\left[(\wbf_i+\wbf_j)^{\!\top}\Xbf(\wbf_i+\wbf_j)-1,\; 2(\wbf_i+\wbf_j)\T\xbf\right]\right\|_2,\\
&(\wbf_i-\wbf_j)^{\!\top}\Xbf(\wbf_i-\wbf_j)+1\geq \nonumber\\
&\qquad\qquad\qquad\qquad\left\|\left[(\wbf_i-\wbf_j)^{\!\top}\Xbf(\wbf_i-\wbf_j)-1,\; 2(\wbf_i-\wbf_j)\T\xbf\right]\right\|_2,
\end{align}
\end{subequations}
which are different from the common-practice SOCP relaxation \eqref{common-socp}.

The proposed parabolic relaxation approach is closely related, yet possesses fundamental differences with other relaxations in the literature. The remainder of the paper articulates major differences of this approach and offers thorough comparisons with  SDP \cite{nesterov1994interior}
and SOCP relaxations \cite{alizadeh2003second}, as well as the diagonally-dominant (DD) and the scaled diagonally-dominant (SDD) approximations \cite{ahmadi2017dsos}. We devote particular attention to the more general case of QCQP with respect to a tensor $\Ybf\in\Rbb^{n\times m}$ for which real-world applications will be discussed in the second part.

\subsection{Parabolic Relaxation for Tensor QCQP}

In this subsection, we discuss a more general case of \eqref{QCQP_vector_a} -- \eqref{QCQP_vector_c}, for which parabolic relaxation enjoys major benefits including a lower number of auxiliary variables in the lifted space compared to the common-practice SDP and SOCP relaxations. Consider the problem of finding a matrix $\Ybf\in\Rbb^{n\times m}$ that minimizes a quadratic objective function subject to a set of equality constraints $\Ecal$ and inequality constraints $\Ical$:
\begin{subequations}
	\begin{align}
		&\underset{\Ybf\in\Rbb^{n\times m}}{\text{minimize}}&& 
		\!\!\!q\Tp_0(\Ybf)  \label{QCQP_tensor_a}\\
		&\text{subject to} && 
		\!\!\!q\Tp_k(\Ybf) = 0, &&&&\!\!\!\!\!\!k\in\Ecal, \label{QCQP_tensor_b}\\
		& && 
		\!\!\!q\Tp_k(\Ybf) \leq 0, &&&&\!\!\!\!\!\!k\in\Ical, \label{QCQP_tensor_c}
	\end{align}
\end{subequations}
where each function $q_k:\Rbb^{n\times m}\to\Rbb$ is defined as $q\Tp_k(\Ybf)\triangleq\tr\{\Ybf\T\!\Abf\Tp_k\Ybf\}+2\,\tr\{\Bbf_k\T \Ybf\}+c\Tp_{k}$, and $\{\Abf_k\in\Sbb_{n},\;\Bbf_k\in\Rbb^{n\times m},\;c_k\in\Rbb\}_{k\in\{0\}\cup\Ecal\cup\Ical}$ are given matrices/scalars. The formulation \eqref{QCQP_tensor_a} -- \eqref{QCQP_tensor_c} is motivated by a wide-variety of applications including problems involving tensor products \cite{zohrizadeh2019non,choi2022multi}. A case study of optimization for dynamical systems is provided in the second part. 

As a major benefit, the proposed approach introduces a single variable per inner-product to relax tensor products, versus a single variable per each scalar product. Indeed, one can cast problem \eqref{QCQP_tensor_a} -- \eqref{QCQP_tensor_c} in a QCQP conical form \eqref{QCQP_vector_a} -- \eqref{QCQP_vector_c} by stacking the elements of $\Ybf$ in a vector and formulating the problem with respect to 
\begin{align}
	[(\Ybf\!\ebf_1)\T,\; (\Ybf\!\ebf_2)\T,\; \ldots,\; (\Ybf\!\ebf_m)\T]^{\top}\in\Rbb^{m\times n},
\end{align}
where $\{\ebf_i\}_{i\in\{1,\ldots,m\}}$ denote the standard basis for $\Rbb^m$. However, preserving the present matrix structure is key in forming an efficient parabolic relaxation. Again, we cast the non-convex problem in the lifted space $\Rbb^{n\times m}\times\Sbb_n$, by introducing a new auxiliary variable $\Xbf\in\Sbb_n$:
\begin{subequations}
	\begin{align}
		&\underset{\begin{subarray}{l}\Ybf\in\Rbb^{n\times m}\\\hspace{-0.2mm}\Xbf\in\Sbb_n\end{subarray}}
		{\text{minimize}}&& 
		\!\!\!\bar{q}\Tp_0(\Ybf\!,\Xbf) \label{QCQP_tensor_lift_a}\\
		&\text{subject to} && 
		\!\!\!\bar{q}\Tp_k(\Ybf\!,\Xbf) = 0, &&&&\!\!\!\!\!\!k\in\Ecal, \label{QCQP_tensor_lift_b}\\
		& && 
		\!\!\!\bar{q}\Tp_k(\Ybf\!,\Xbf) \leq 0, &&&&\!\!\!\!\!\!k\in\Ical, \label{QCQP_tensor_lift_c}\\
		& && 
		\!\!\!\Xbf=\Ybf\Ybf\T, &&&&\label{QCQP_tensor_lift_d}
	\end{align}
\end{subequations}
where $\bar{q}\Tp_k(\Ybf\!,\Xbf)\triangleq\tr\{\Abf\Tp_k\Xbf\}+2\,\tr\{\Bbf_k\T \Ybf\}+c\Tp_{k}$ for each $k\in\{0\}\cup\Ecal\cup\Ical$. Observe that for each $i,j\in\Ncal$, the $(i,j)$-element of the matrix $\Xbf$ accounts for the inner product of rows $i$ and $j$ of $\Ybf$, i.e.,
\begin{align}
	\Xbf=\Ybf\Ybf\T
	\quad \Leftrightarrow\quad 
	X_{ij}=\la\, \ebf_i\T \Ybf\, ,\, \ebf_j\T \Ybf\,\ra, \quad \forall i,j\in\Ncal.
\end{align}
Let $\Fcal^{\mathrm{lifted}}$ denote the feasible sets of
\eqref{QCQP_tensor_lift_a} -- \eqref{QCQP_tensor_lift_d}.

Now, relaxing the single non-convex constraint \eqref{QCQP_tensor_lift_d} to parabolic constraints as before leads to a parabolic relaxation of problem \eqref{QCQP_tensor_a} -- \eqref{QCQP_tensor_a}:
\begin{subequations}
	\begin{align}
		&\underset{\begin{subarray}{l}\Ybf\in\Rbb^{n\times m}\\\hspace{-0.2mm}\Xbf\in\Sbb_n\end{subarray}}
		{\text{minimize}}&& 
		\!\!\!\bar{q}\Tp_0(\Ybf\!,\Xbf)\label{QCQP_tensor_para_a}\\
		&\text{subject to} && 
		\!\!\!\bar{q}\Tp_k(\Ybf\!,\Xbf) = 0, &&&&\!\!\!\!\!\!k\in\Ecal,\label{QCQP_tensor_para_b}\\
		& && 
		\!\!\!\bar{q}\Tp_k(\Ybf\!,\Xbf) \leq 0, &&&&\!\!\!\!\!\!k\in\Ical,\label{QCQP_tensor_para_c}\\
		& && 
		\!\!\!(\wbf_i+\wbf_j)^{\!\top}\Xbf(\wbf_i+\wbf_j)\geq \|(\wbf_i+\wbf_j)\T\Ybf\|_2^2, &&&&
		\!\!\!\!\!\!(i,j)\in\Ncal,\label{QCQP_tensor_para_d}\\
		& && 
		\!\!\!(\wbf_i-\wbf_j)^{\!\top}\Xbf(\wbf_i-\wbf_j)\geq \|(\wbf_i-\wbf_j)\T\Ybf\|_2^2, &&&&
		\!\!\!\!\!\!(i,j)\in\Ncal,\label{QCQP_tensor_para_e}
	\end{align}
\end{subequations}
where again, $\{\wbf_i\}_{i\in\Ncal}$ is an arbitrary basis for $\Rbb^n$.  
Let $\Fcal_{\Pcal}$ denote the feasible set of \eqref{QCQP_tensor_para_a} -- \eqref{QCQP_tensor_para_e}. 
The dual problem for \eqref{QCQP_tensor_para_a} -- \eqref{QCQP_tensor_para_e} is:
\begin{subequations}
	\begin{align}
		&\underset{\begin{subarray}{l}
				\taubf\in\Rbb^{|\Ecal\cup\Ical|}\\
				\Tbf\in\Sbb_m
		\end{subarray}}
		{\text{maximize}}
		&&\tr\{\Tbf\}\label{QCQP_tensor_para_dual_a}\\
		&\text{subject to} 
		&&\begin{bmatrix}
			\Abf_0 & \Bbf_0 \\
			\Bbf_0\T & -\Tbf+m^{-1}c_0\Ibf_{\! m}\\
		\end{bmatrix}+
		\sum_{k\in\Ecal\cup\Ical}{\tau_k
			\begin{bmatrix}
				\Abf_k & \Bbf_k \\
				\Bbf_k\T & m^{-1}c_k\Ibf_{\! m}\\
		\end{bmatrix}}\succeq 0,\label{QCQP_tensor_para_dual_b}\\
		& &&[\wbf_1,\ldots,\wbf_n]\T\left(\Abf_0+\sum_{k\in\Ecal\cup\Ical}{\tau_k\Abf_k}\right)[\wbf_1,\ldots,\wbf_n]\in \Dbb_n,\label{QCQP_tensor_para_dual_c}\\
		& && \tau_k \geq 0, &&&&\hspace{-2cm} k\in\Ical.\label{QCQP_tensor_para_dual_d}
	\end{align}
\end{subequations}
As we will demonstrate in Section \ref{coord}, the quality of parabolic relaxation is dependent on the choice of the basis.
However, a simple, yet effective, choice for $\{\wbf_i\}_{i\in\Ncal}$ is the standard basis for $\Rbb^n$, which leads to the following convex quadratic inequalities:
\begin{subequations}
	\begin{align}
		&\!\!\!\! X_{ii}+ X_{jj} + 2 X_{ij}\geq \| (\ebf_i+\ebf_j)\T\Ybf\|^2_2,\qquad\qquad \forall i,j\in\Ncal.\label{QCQP_tensor_para_stan_d}\\		
		&\!\!\!\! X_{ii}+ X_{jj} - 2 X_{ij}\geq \| (\ebf_i-\ebf_j)\T\Ybf\|^2_2,\qquad\qquad \forall i,j\in\Ncal,\label{QCQP_tensor_para_stan_e}
	\end{align}
\end{subequations}
Section \ref{sec:related} offers thorough comparisons between this approach and related methods for handling the nonconvex constraint \eqref{QCQP_tensor_lift_d}. 

\section{Comprehensive Boundary Property}\label{cbp}

Polynomial-time solvable convex relaxations may be inexact when applied to general non-convex optimization problems, in which case, the resulting solutions may be infeasible for the original problem. In order to make a meaningful assessment of the proposed parabolic relaxation, we first introduce a desirable property, namely the {comprehensive boundary property} (CBP), which is useful when one needs to infer feasible points from inexact relaxations. The CBP ensures that every extreme point which can be an optimal solution of the non-convex problem 
is included on the boundary of the relaxed feasible set, which makes it attainable for an alternative objective function. This notion is formally defined next.
\begin{definition}[Comprehensive Boundary Property] Let $\Fcal^{\mathrm{relax}}\subseteq\Rbb^{n\times m}\times\Sbb_n$ be a closed, convex relaxation of $\Fcal^{\mathrm{lifted}}$. $\Fcal^{\mathrm{relax}}$ is said to satisfy the \textit{comprehensive boundary property} for $\Fcal^{\mathrm{lifted}}$ if $\Fcal^{\mathrm{lifted}}\subseteq\partial\Fcal^{\mathrm{relax}}$.
\label{cbp}\end{definition}

Definition \ref{cbp} is illustrated in Figure \ref{fig_CBP}.A. As shown in Figure \ref{fig_CBP}.B, CBP allows us to round an infeasible point $\Ych\in\Rbb^{n\times m}\setminus\Fcal$ into a member of $\Fcal$ by incorporating an appropriate linear penalty term into the objective. 

Moreover, as we will demonstrate later, CBP allows us to search for solutions by solving a sequence of convex problems in a manner that membership to $\Fcal$ is preserved and the objective is monotonically improved. The success of the sequential framework and penalization in solving bilinear matrix inequalities (BMIs) demonstrated in \cite{ibaraki2001rank,BMI1,BMI2}. The incorporation of penalty terms into the objective of SDP relaxations are proven to be computationally effective for solving non-convex optimization problems in power systems \cite{madani2015convex,madani2016promises,OPF_CDC,UC_CDC}. These papers show that penalizing certain physical quantities in power network optimization problems such as reactive power loss or thermal loss facilitates the recovery of feasible points from convex relaxations. In \cite{ibaraki2001rank}, a sequential framework is introduced for solving BMIs without theoretical guarantees. This approach is investigated further and theoretical results for the penalized SDP are derived \cite{BMI1,BMI2,madanipolynomial}.  

Below we show that while the parabolic relaxation of $\Fcal^{\mathrm{lifted}}$ enjoys CBP, 
the SOCP relaxation does not.

\begin{theorem}\label{thm_CBP}
$\Fcal_{\Pcal}$ satisfies the comprehensive boundary property for $\Fcal^{\mathrm{lifted}}$.
\end{theorem}

\begin{proof} 
Recall that $\Fcal^{\mathrm{lifted}}$ and $\Fcal_{\Pcal}$ denote the feasible sets of
\eqref{QCQP_tensor_lift_a} -- \eqref{QCQP_tensor_lift_d} and
\eqref{QCQP_tensor_para_a} -- \eqref{QCQP_tensor_para_e}, respectively.
Consider an arbitrary $(\Ych,\Ych\Ych\T)\in\Fcal^{\mathrm{lifted}}$. In order to prove $(\Ych,\Ych\Ych\T)\in\partial\Fcal_{\Pcal}$, it suffices to construct a linear objective function with $(\Ych,\Ych\Ych\T)$ as the optimal solution. The following objective serves this purpose: 
	\begin{subequations}
		\begin{align}
			&\underset{\begin{subarray}{l}\Ybf\in\Rbb^{n\times m}\\\hspace{-0.2mm}\Xbf\in\Sbb_n\end{subarray}}
			{\text{minimize}}&& 
			\!\!\!\tr\{\Xbf-2\Ych\Ybf\T+\Ych\Ych\T\}\label{QCQP_tensor_para_aux_a}\\
			&\text{subject to} && 
			\!\!\!(\Ybf,\Xbf)\in\Fcal_{\Rcal}, &&&&\label{QCQP_tensor_para_aux_b}
		\end{align}
	\end{subequations}
	where $\Fcal_{\Rcal}$ denotes the set of pairs $(\Ybf,\Xbf)\in\Rbb^{n\times m}\times\Sbb_n$ that satisfy \eqref{QCQP_tensor_para_b} -- \eqref{QCQP_tensor_para_e}.
	The corresponding dual problem is:
	\begin{subequations}
		\begin{align}
			&\underset{\begin{subarray}{l}
					\taubf\in\Rbb^{|\Ecal\cup\Ical|}\\
					\Tbf\in\Sbb_m
			\end{subarray}}
			{\text{maximize}}
			&&\tr\{\Tbf\}\label{QCQP_tensor_para_dual_aux_a}\\
			&\text{subject to} 
			&&\begin{bmatrix}
				\Ibf_n & \Ych \\
				\Ych\T & -\Tbf+m^{-1}\tr\{\Ych\T\Ych\}\Ibf_{\!m}\\
			\end{bmatrix}+
			\sum_{k\in\Ecal\cup\Ical}{\tau_k
				\begin{bmatrix}
					\Abf_k & \Bbf_k \\
					\Bbf_k\T & m^{-1}c_k\Ibf_{\! m}\\
			\end{bmatrix}}\succeq 0,\label{QCQP_tensor_para_dual_aux_b}\\
			& &&[\wbf_1,\ldots,\wbf_n]\T\left(\Ibf_n+\sum_{k\in\Ecal\cup\Ical}{\tau_k\Abf_k}\right)[\wbf_1,\ldots,\wbf_n]\in \Dbb_n,\label{QCQP_tensor_para_dual_aux_c}\\
			& && \tau_k \geq 0, &&&&\hspace{-2cm} k\in\Ical.\label{QCQP_tensor_para_dual_aux_d}
		\end{align}
	\end{subequations}
	If $\varepsilon>0$ is sufficiently small, then  $(\taubf,\Tbf)=\big(\varepsilon\,\boldsymbol{1}_{|\Ecal\cup\Ical|},\;-\varepsilon^{-1}\,\boldsymbol{I}_{m}\big)$ is strictly feasible for the dual problem which proves strong duality.
	Now, it is straightforward to verify that 
	$\big(\Ych,\Ych\Ych\T\big)$ and $\big(\boldsymbol{0}_{|\Ecal\cup\Ical|},-\Ych\T\Ych+m^{-1}\tr\{\Ych\T\Ych\}\Ibf_{\!m}\big)$ are a pair of primal and dual optimal solutions, which concludes the proof.
	\qed\end{proof}

 \begin{figure*}[t]
	\centering
	\includegraphics[height=0.28\textwidth]{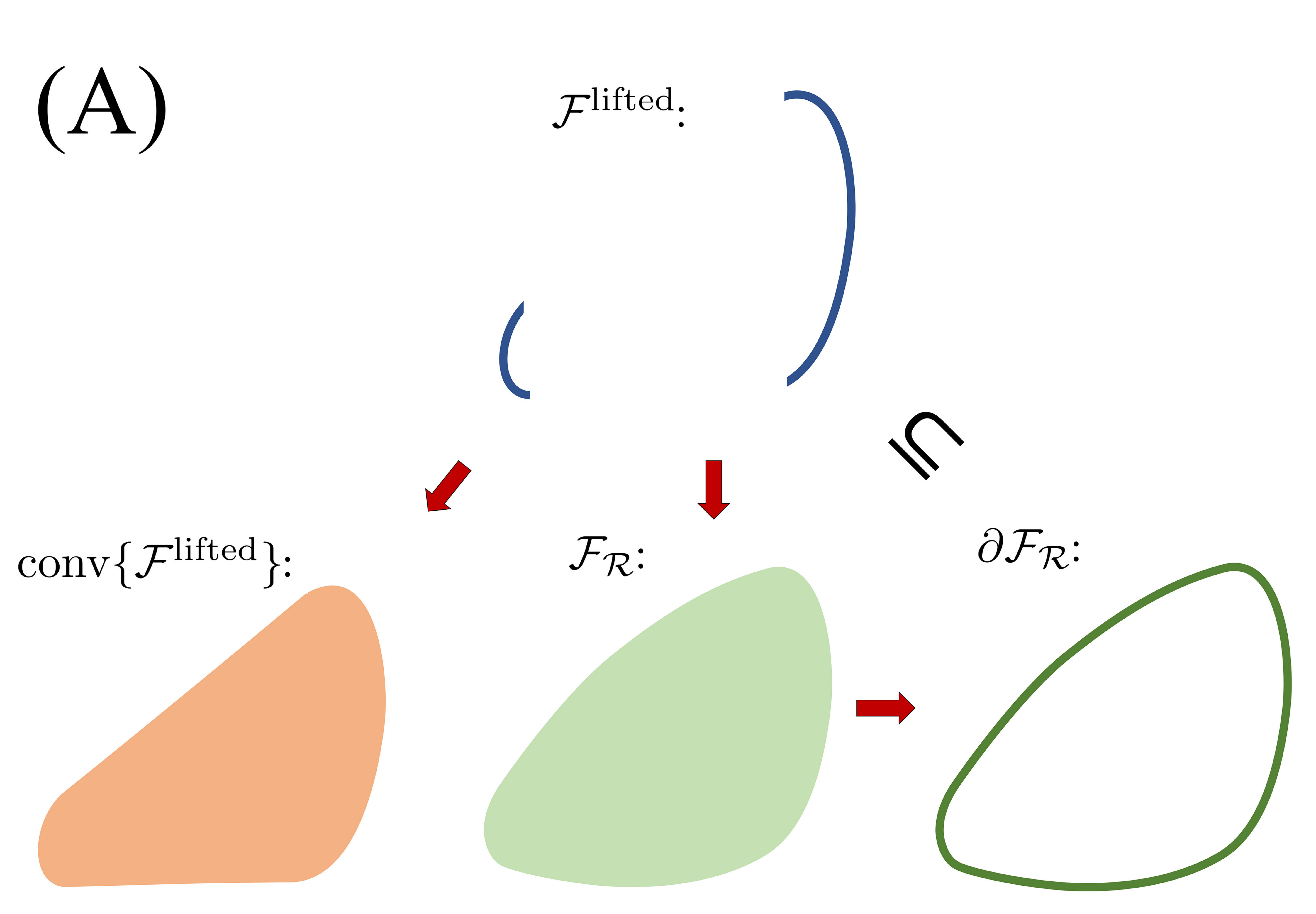}\hspace{1cm}
	\includegraphics[height=0.28\textwidth]{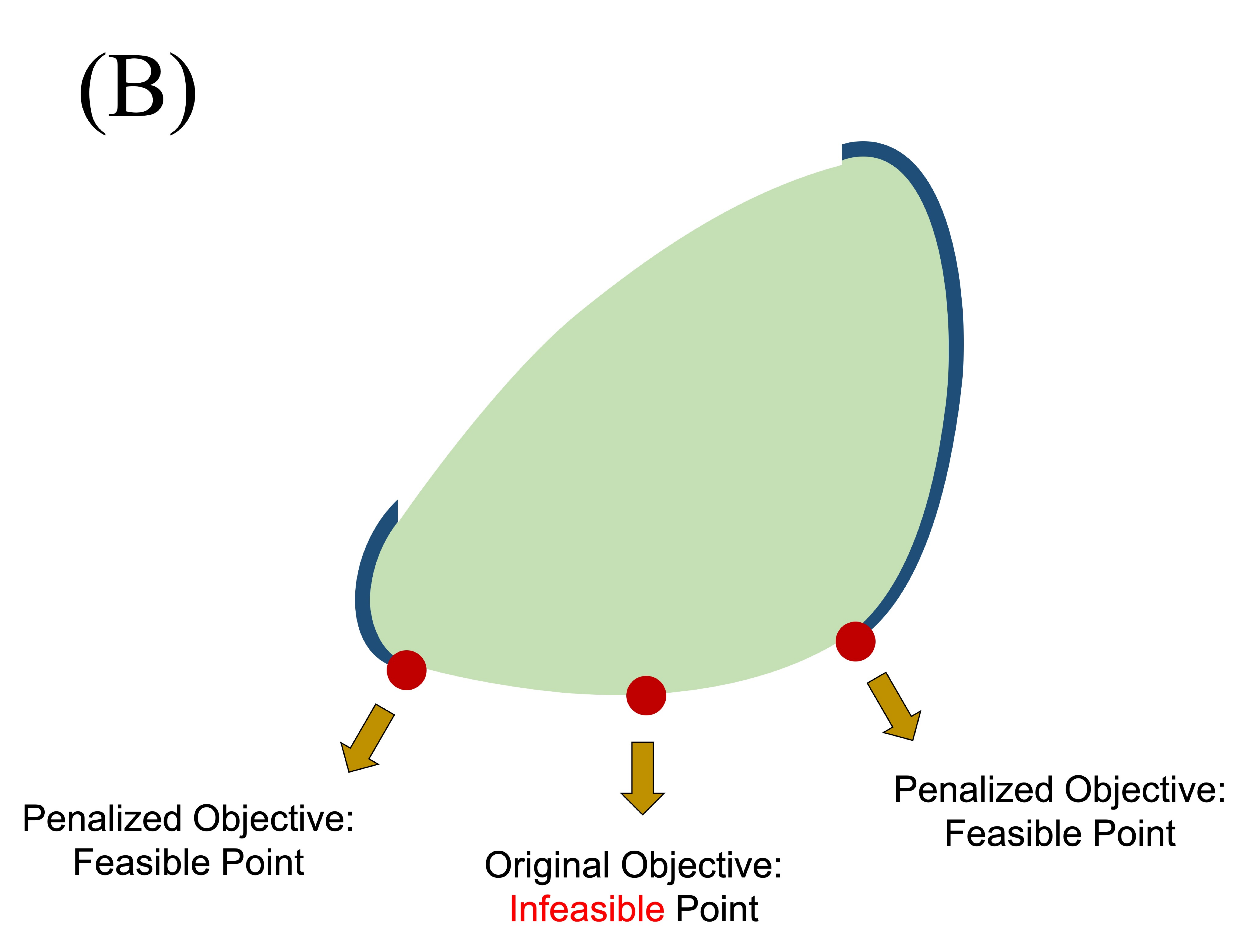}
	\caption{
		(A) Comprehensive boundary property for convex set $\Rcal$, 
		(B) Penalization.
	}
	\label{fig_CBP}
\end{figure*}

In the following example, we show that the SOCP relaxation does not necessarily satisfy CBP, contrasting it with the parabolic relaxation.

\begin{example}
Consider a two-dimensional feasible set $\Fcal=\{(x_1,x_2)\in\Rbb^2\,|\,x_2+(x_1-x_2)^2=0\}$. The resulting feasible sets for the parabolic relaxation from the standard basis \eqref{eq:standard-parabolic}
and the SOCP relaxation \eqref{common-socp} are given as:
\begin{subequations}
\begin{align}
\mathrm{proj}_{\Rbb^2}\{\Fcal_{\Pcal}\}=&\Big\{(x_1,x_2)\in\Rbb^2\,|\,\exists(X_{11},X_{22},X_{12})\in\Rbb^3: \nonumber\\
&x_2+X_{11}+X_{22}-2X_{12}=0,\,
X_{11}+X_{22}+2X_{12}\geq(x_1+x_2)^2,\nonumber\\
&X_{11}+X_{22}-2X_{12}\geq(x_1-x_2)^2,
\,X_{11}\geq x_1^2,
\,X_{22}\geq x_2^2\Big\}  \nonumber \\
= &\big \{ (x_1,x_2)\in\Rbb^2: x_2+(x_1-x_2)^2 \le 0 \big \}, \nonumber \\
\mathrm{proj}_{\Rbb^2}\{\Fcal_{\mathrm{SOCP}}\}= & \Big\{(x_1,x_2)\in\Rbb^2\,|\,\exists(X_{11},X_{22},X_{12})\in\Rbb^3: \nonumber\\
&x_2+X_{11}+X_{22}-2X_{12}=0,\,
X_{11}X_{22}\geq X_{12}^2,\nonumber\\
&X_{11}\geq x_1^2,
\,X_{22}\geq x_2^2\Big\} \nonumber \\ 
= & \big \{  (x_1,x_2)\in\Rbb^2: x_2 \le 0 \big \}, \nonumber
\end{align}
\end{subequations}
as demonstrated in Figure \ref{fig_CBP_SOCP}, where for every $\Gcal \subseteq \Rbb^n$, we define 
$\mathrm{proj}_{\Rbb^2}\{\Gcal\}\triangleq\{(x_1,x_2)\in\Rbb^2\,|\,\xbf\in\Gcal\}$.
While parabolic relaxation enjoys CBP, the SOCP relaxation does not enjoy this property.

\begin{figure*}[t]
	\centering
	\includegraphics[height=0.45\textwidth]{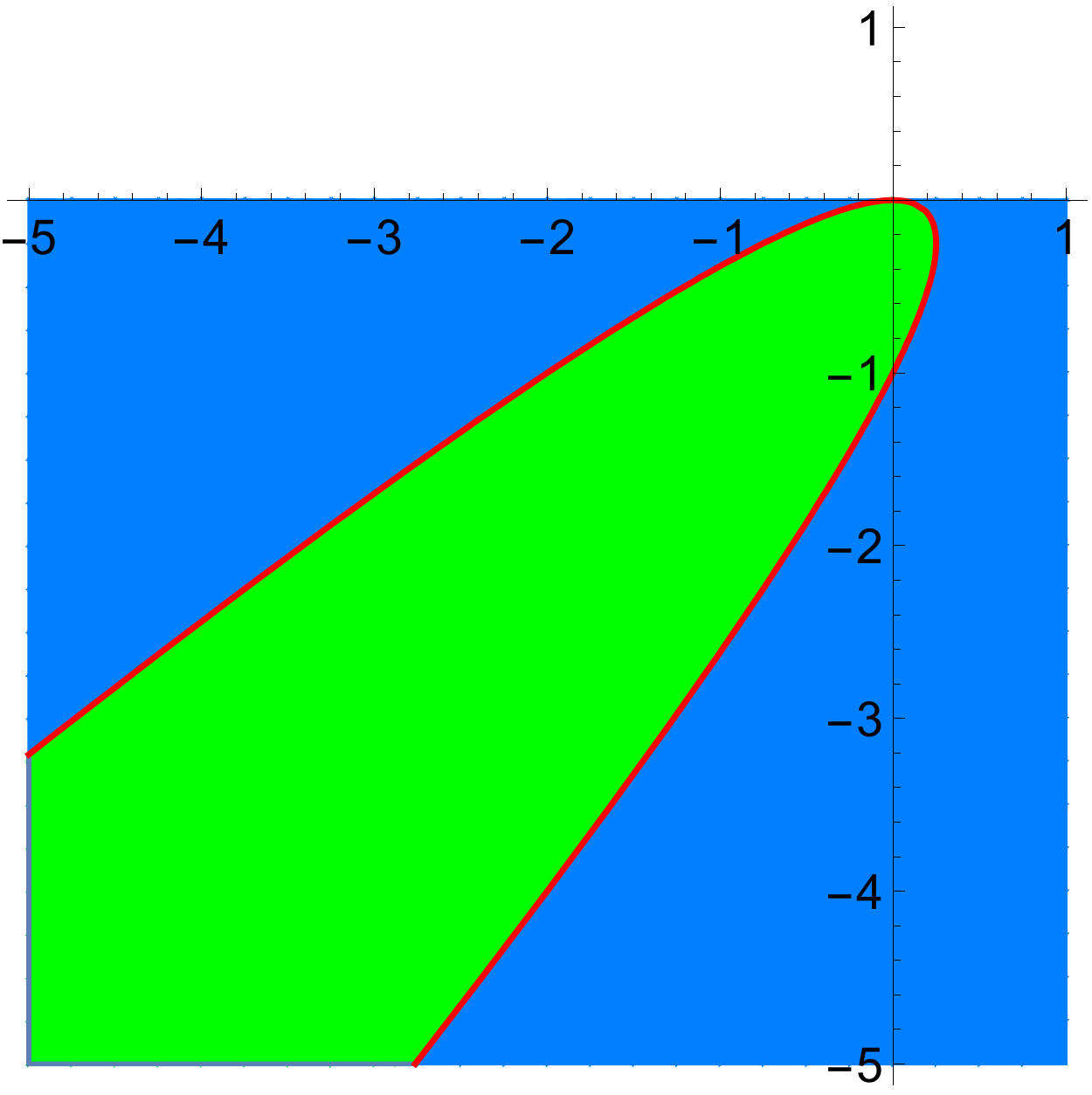}
	\caption{
The absence of CBP for the SOCP relaxation. The red curve demonstrates the feasible set $\Fcal=\{(x_1,x_2)\in\Rbb^2\,|\,x_2+(x_1-x_2)^2=0\}$, while the green and blue regions show the resulting parabolic and SOCP relaxations, respectively.
	}
	\label{fig_CBP_SOCP}
\end{figure*}

This example also illustrates that the SOCP relaxation can be much weaker than the parabolic relaxation
from the standard basis. The computational studies in part two of this work \cite{parabolic_part2} show that, indeed, for most of the instances, the parabolic relaxation from the standard basis is far superior compared to the SOCP relaxation; however, it does not necessarily dominate the SOCP relaxation in general. 
\end{example}


\section{Sequential Penalized Parabolic Relaxation}
In order to recover feasible points for cases where parabolic relaxation is not exact, one can incorporate an additional linear term into the objective function to arrive at a penalized convex problem. Given $\Ych\in\Rbb^{n\times m}$ as an initial point and a constant $\eta>0$, the penalized convex relaxation is defined as
\begin{subequations}
	\begin{align}
		&\underset{\begin{subarray}{l}\Ybf\in\Rbb^{n\times m}\\\hspace{-0.2mm}\Xbf\in\Sbb_n\end{subarray}}
		{\text{minimize}}&& 
		\!\!\!\bar{q}\Tp_0(\Ybf\!,\Xbf) +\eta\times \tr\{\Xbf-2\Ych\Ybf\T+\Ych\Ych\T\}\label{prob_relax_pen_1a}\\
		&\text{subject to} && 
		\!\!\!\bar{q}\Tp_k(\Ybf\!,\Xbf) = 0,&&&&\!\!\!\!\!\!k\in\Ecal,\label{prob_relax_pen_1b}\\
		& && 
		\!\!\!\bar{q}\Tp_k(\Ybf\!,\Xbf) \leq 0,&&&&\!\!\!\!\!\!k\in\Ical,\label{prob_relax_pen_1c}\\
		& && 
		\!\!\!X_{ii}+ X_{jj}-2 X_{ij}\geq \| (\ebf_i-\ebf_j)\T\Ybf\|^2_2,&&&&
		\!\!\!\!\!\! i,j\in\Ncal,\label{prob_relax_pen_1d}\\
		& && 
		\!\!\!X_{ii}+ X_{jj}+2 X_{ij}\geq \|(\ebf_i+\ebf_j)\T\Ybf\|^2_2.&&&&
		\!\!\!\!\!\! i,j\in\Ncal,\label{prob_relax_pen_1e}
	\end{align}
\end{subequations}
The following two theorems guarantee that an appropriate choice of $\Ych$ and $\eta$ results in a feasible point for the non-convex problem \eqref{QCQP_tensor_a} -- \eqref{QCQP_tensor_c}. In the second part, we prove that the convex problem \eqref{prob_relax_pen_1a} -- \eqref{prob_relax_pen_1e} has a unique solution $(\Yst,\Xst)$ that satisfies $\Xst=\Yst\Yst\T$ for appropriate choices of $\eta$ and under either of the following conditions:
\begin{itemize}
	\item[i)] If $\Ych$ is a feasible point for \eqref{QCQP_tensor_a} -- \eqref{QCQP_tensor_c} that satisfies linear independence constraint qualification (LICQ), or
	\item[ii)] If $\Ych\in\Rbb^{n\times m}$ is close to the feasible set of \eqref{QCQP_tensor_a} -- \eqref{QCQP_tensor_c} and satisfies a generalized LICQ (GLICQ) property (\cite[Definition 3]{parabolic_part2}). 
\end{itemize}
An accurate statement of these results, along with explicit bounds for $\eta$ are given by \cite[Theorems 1 and 2]{parabolic_part2}. In light of these theorems, we adopt a sequential framework to obtain feasible near optimal points for  problem \eqref{QCQP_tensor_a} -- \eqref{QCQP_tensor_c}. This procedure is detailed in Algorithm \ref{alg:1}, for which \cite[Theorem 3]{parabolic_part2} offers a proof of convergence.

\begin{algorithm}[t]
	\caption{Sequential Penalized Parabolic Relaxation.}\label{alg:1}
	\begin{algorithmic}[1]
		\Require {$\Ybf^{(0)}\in\Rbb^{n\times m}$, $\eta>0$ and $\ell:=0$}
		\Repeat
		\State $\ell:=\ell+1$
		\State solve the penalized convex problem \eqref{prob_relax_pen_1a} -- \eqref{prob_relax_pen_1e} with $\Ych=\Ybf^{(\ell-1)}$ to obtain $\big(\Yst,\Xst\big)$
		\State $\Ybf^{(\ell)}:=\Yst$
		\Until {stopping criterion is met.}
		\Ensure {$\Ybf^{(\ell)}$}
	\end{algorithmic}\label{al:alg_1}
\end{algorithm}

\section{SDP Strength via Basis Update}\label{coord}
In this section, we show that an appropriate choice of basis can result in a parabolic relaxation that is as strong as SDP equipped with additional valid inequalities in the direction of desired objective. To this end, consider the following SDP relaxation of problem \eqref{QCQP_tensor_a} -- \eqref{QCQP_tensor_c}:
\begin{subequations}
	\begin{align}
		&\underset{\begin{subarray}{l}\Ybf\in\Rbb^{n\times m}\\\hspace{-0.2mm}\Xbf\in\Sbb_n\end{subarray}}
		{\text{minimize}}&& 
		\!\!\!\phantom{\Lambdabf}\phantom{\tau_k}\;\;\;\;\bar{q}\Tp_0(\Ybf\!,\Xbf) \label{prob_sdp_rlt_1a}\\
		&\text{subject to} && 
		\!\!\!\phantom{\Lambdabf}\tau_k:\;\bar{q}\Tp_k(\Ybf\!,\Xbf) = 0,&&&&\!\!\!\!\!\!k\in\Ecal,\label{prob_sdp_rlt_1b}\\
		& && 
		\!\!\!\phantom{\Lambdabf}\tau_k:\;\bar{q}\Tp_k(\Ybf\!,\Xbf) \leq 0,&&&&\!\!\!\!\!\!k\in\Ical,\label{prob_sdp_rlt_1c}\\
		& && 
		\!\!\!\phantom{\Lambdabf}\nu_k:\;\bar{r}\Tp_k(\Ybf\!,\Xbf) \leq 0,&&&&\!\!\!\!\!\!k\in\Vcal,\label{prob_sdp_rlt_1d}\\
		& && 
		\!\!\!\phantom{\tau_k\Lambdabf:\;}\;\Xbf\succeq\Ybf\Ybf\T,&&&&\label{prob_sdp_rlt_1e}
	\end{align}
\end{subequations}
where $\taubf\in\Rbb^{|\Ecal\cup\Ical|}$ and $\nubf\in\Rbb^{|\Vcal|}$ are Lagrange multipliers, and 
$\bar{r}\Tp_k(\Ybf\!,\Xbf)\triangleq\tr\{\Fbf\Tp_k\Xbf\}+2\,\tr\{\Gbf_k\T \Ybf\}+h\Tp_{k}$ for each $k\in\Vcal$, where $\{\Fbf_k\in\Sbb_{n},\;\Gbf_k\in\Rbb^{n\times m},\;h_k\in\Rbb\}_{k\in\Vcal}$ are given matrices/scalars.
Constraint \eqref{prob_sdp_rlt_1d} may include any arbitrary collection of affine valid inequalities such as the reformulation-linearization (RLT) inequalities. The corresponding dual SDP relaxation is formulated as:
\begin{subequations}
	\begin{align}
		&\underset{\begin{subarray}{l}\taubf\in\Rbb^{|\Ecal\cup\Ical|}\\
				\nubf\in\Rbb^{|\Vcal|}\\
				\!\Tbf\in\Sbb_m\end{subarray}}
		{\text{maximize}}
		&&\tr\{\Tbf\}\label{prob_sdp_rlt_dual_1a}\\
		&\text{subject to} 
		&&\begin{bmatrix}
			\Abf_0 & \Bbf_0 \\
			\Bbf_0\T & -\Tbf+m^{-1}c_0\Ibf_{\! m}\\
		\end{bmatrix}+
		\sum_{k\in\Ecal\cup\Ical}{\tau_k
			\begin{bmatrix}
				\Abf_k & \Bbf_k \\
				\Bbf_k\T & m^{-1}c_k\Ibf_{\! m}\\
		\end{bmatrix}}+ \nonumber\\
		& && \ \ \ \ \ \ \ 	\sum_{k\in\Vcal}{\nu_k
			\begin{bmatrix}
				\!\Fbf_{\!k} & \Gbf_{k} \\
				\Gbf_{k}\T & m^{-1}h_k\Ibf_{\! m}\\
		\end{bmatrix}}\succeq 0,\label{prob_sdp_rlt_dual_1b}\\
	   & && \tau_k \geq 0, \ \ k\in\Ical, 
		\ \ \ \   \nu_k \geq 0, \ \ k\in\Vcal.\label{prob_sdp_rlt_dual_1d}
	\end{align}
\end{subequations}
Next, we show that any dual feasible point $(\hat{\taubf},\hat{\nubf},\hat{\Tbf})\in\Rbb^{|\Ecal\cup\Ical|}\!\times\!\Rbb^{|\Vcal|}\!\times\!\Sbb_m$ for \eqref{prob_sdp_rlt_dual_1a} -- \eqref{prob_sdp_rlt_dual_1d} can be used to construct a
basis $\{\wht_i\}_{i\in\Ncal}$ for $\Rbb^n$, resulting in a parabolic relaxation of the form:
\begin{subequations}
	\begin{align}
		&\underset{\begin{subarray}{l}\Ybf\in\Rbb^{n\times m}\\\hspace{-0.2mm}\Xbf\in\Sbb_n\end{subarray}}
		{\text{minimize}}&& 
		\!\!\!\bar{q}\Tp_0(\Ybf\!,\Xbf)\label{prob_cor_1a}\\
		&\text{subject to} && 
		\!\!\!\bar{q}\Tp_k(\Ybf\!,\Xbf) = 0,&&&&\!\!\!\!\!\!k\in\Ecal,\label{prob_cor_1b}\\
		& && 
		\!\!\!\bar{q}\Tp_k(\Ybf\!,\Xbf) \leq 0,&&&&\!\!\!\!\!\!k\in\Ical,\label{prob_cor_1c}\\
		& && 
		\!\!\!\!\sum_{k\in\Vcal}{\nuht_k\Tp\bar{r}\Tp_k(\Ybf\!,\Xbf)} \leq 0,&&&&\label{prob_cor_1d}\\
		& && 
		\!\!\!  (\wht_i-\wht_j)\T\Xbf(\wht_i-\wht_j) \geq \| (\wht_i-\wht_j)\T \Ybf\|^2_2,&&&&
		\!\!\!\!\!\! i,j\in\Ncal,\label{prob_cor_1e}\\
		& && 
		\!\!\!  (\wht_i+\wht_j)\T\Xbf(\wht_i+\wht_j) \geq \| (\wht_i+\wht_j)\T \Ybf\|^2_2,&&&&
		\!\!\!\!\!\! i,j\in\Ncal,\label{prob_cor_1f}
	\end{align}
\end{subequations}
whose lower-bound is not less than $\tr\{\hat{\Tbf}\}$. This approach is useful for cases where one needs to solve a sequence of convex surrogates given by Algorithm \ref{alg:1} or a branch-and-bound search.

\begin{theorem}\label{thm_cord}
	Let $(\hat{\taubf},\hat{\nubf},\hat{\Tbf})\in\Rbb^{|\Ecal\cup\Ical|}\!\times\!\Rbb^{|\Vcal|}\!\times\!\Sbb_m$ be an arbitrary feasible point for the SDP \eqref{prob_sdp_rlt_dual_1a} -- \eqref{prob_sdp_rlt_dual_1d} and let $\{\wht_i\}_{i\in\Ncal}$ be a basis for $\Rbb^n$ such that:
	\begin{align}
		[\wht_1,\wht_2,\ldots,\wht_n]\T\left(\Abf_0\!+\!\!\!\!\sum_{k\in\Ecal\cup\Ical}{\hat{\tau}_k\Abf_k}\!+\!\sum_{k\in\Vcal}{\hat{\nu}_k\Fbf_{\!k}}\right)
		[\wht_1,\wht_2,\ldots,\wht_n]\in\Dbb_n.\!\!\! \label{DDcond}
	\end{align}
	Every feasible point $(\Ybf\!,\Xbf)\in\Rbb^{n\times m}\times\Sbb_{n}$ of the parabolic relaxation \eqref{prob_cor_1a} -- \eqref{prob_cor_1f} satisfies:
	\begin{align}
		\bar{q}_0(\Ybf\!,\Xbf)\geq\tr\{\hat{\Tbf}\}. \label{dbound}
	\end{align}
\end{theorem}

\begin{proof}
To prove this theorem, we formulate the dual of parabolic relaxation \ref{prob_cor_1a} -- \ref{prob_cor_1f}:
	\begin{subequations}
		\begin{align}
			&\underset{\begin{subarray}{l}\taubf\in\Rbb^{|\Ecal\cup\Ical|}\\
					\kappa\in\Rbb\\
					\!\Tbf\in\Sbb_m\end{subarray}}
			{\text{maximize}}
			&&\tr\{\Tbf\}\label{par_dual_a}\\
			&\text{subject to} 
			&&\begin{bmatrix}
				\Abf_0 & \Bbf_0 \\
				\Bbf_0\T & -\Tbf+m^{-1}c_0\Ibf_{\! m}\\
			\end{bmatrix}+
			\sum_{k\in\Ecal\cup\Ical}{\tau_k
				\begin{bmatrix}
					\Abf_k & \Bbf_k \\
					\Bbf_k\T & m^{-1}c_k\Ibf_{\! m}\\
			\end{bmatrix}}+ \nonumber\\
			& && \quad\quad	\kappa\sum_{k\in\Vcal}{\nuht_k
				\begin{bmatrix}
					\!\Fbf_{\!k} & \Gbf_{k} \\
					\Gbf_{k}\T & m^{-1}h_k\Ibf_{\! m}\\
			\end{bmatrix}}\succeq 0,\label{par_dual_b}\\
			& &&[\wbf_1,\ldots,\wbf_n]\T\left(\Abf_0\!+\!\sum_{k\in\Ecal\cup\Ical}{\tau_k\Abf_k}\!+\!\kappa\sum_{k\in\Vcal}{\nuht_k\Tp\Fbf_{\!k}}\right)[\wbf_1,\ldots,\wbf_n]\!\in\! \Dbb_n,\label{par_dual_c}\\
			& && \tau_k \geq 0, &&&&\hspace{-6cm} k\in\Ical,\label{par_dual_d}\\
			& && \kappa_{\;}\geq0,\label{par_dual_e}
		\end{align}
	\end{subequations}
	with $\kappa\in\Rbb$ representing the Lagrange multiplier associated with constraint \ref{prob_cor_1d}.
	It is straightforward to verify that $(\taubf,\kappa,\Tbf)=(\hat{\taubf},1,\hat{\Tbf})$ is a feasible point for \eqref{par_dual_a} -- \eqref{par_dual_e}, which implies \eqref{dbound} according to weak duality.
\qed\end{proof}


Note that for every dual feasible point $(\hat{\taubf},\hat{\nubf},\hat{\Tbf})\in\Rbb^{|\Ecal\cup\Ical|}\!\times\!\Rbb^{|\Vcal|}\!\times\!\Sbb_m$, according to the constraint \eqref{prob_sdp_rlt_dual_1b}, the matrix $\Abf_0+\sum_{k\in\Ecal\cup\Ical}{\hat{\tau}_k\Abf_k}+\sum_{k\in\Vcal}{\hat{\nu}_k\Fbf_{\!k}}$ is positive semidefinite and its eigenspace is a candidate for $\{\wht_i\}_{i\in\Ncal}$ which satisfies \eqref{DDcond}.
\begin{corollary}
	Let $(\hat{\Ybf},\hat{\Xbf})\in\Rbb^{n\times m}\!\times\!\Sbb_m$ and $(\hat{\taubf},\hat{\nubf},\hat{\Tbf})\in\Rbb^{|\Ecal\cup\Ical|}\!\times\!\Rbb^{|\Vcal|}\!\times\!\Sbb_m$ denote a  pair of primal and dual optimal solutions for the SDP relaxation \eqref{prob_sdp_rlt_1a} -- \eqref{prob_sdp_rlt_1d} and its dual \eqref{prob_sdp_rlt_dual_1a} -- \eqref{prob_sdp_rlt_dual_1d}. Assume that \eqref{DDcond} is satisfied and strong duality holds between \eqref{prob_sdp_rlt_1a} -- \eqref{prob_sdp_rlt_1d} and \eqref{prob_sdp_rlt_dual_1a} -- \eqref{prob_sdp_rlt_dual_1d}. If $(\Ybf\!,\Xbf)\in\Rbb^{n\times m}\!\times\!\Sbb_m$ is an optimal solution for \eqref{prob_cor_1a} -- \eqref{prob_cor_1f}, then:
	\begin{align}
		\bar{q}_0(\Ybf\!,\Xbf)=\bar{q}_0({\hat{\Ybf}}\!,{\hat{\Xbf}}), \label{dbound2}
	\end{align}
	i.e., the parabolic relaxation with SDP-induced basis is as strong as SDP relaxation for the minimization of $\bar{q}_0$.
\end{corollary}

While it may seem counter-intuitive to solve an SDP to achieve the strength of SDP bounds through the parabolic relaxation, the potential advantage is highlighted for algorithms that require the solution of a sequence of similar relaxations, e.g., branch-and-bound or penalty reformulation methods. After a one-time computation of an optimal basis, the cheaper-to-compute parabolic relaxation can replace costly semidefinite programming (SDP) relaxations.



\section{Comparisons with Related Relaxations, Restrictions, and Approximations}\label{sec:related}
In this section, we discuss the relation between the parabolic relaxation and other approaches. We cover the well-known SDP relaxation, as well as:
\begin{itemize}
	\item {\it Outer-approximations of SDP:} Two variants of SOCP relaxation as well as DD restriction of the dual space. 
	\item {\it Inner-approximations {of SDP}:} DD and SDD approximation methods. 
\end{itemize}
We show that in many ways, the parabolic relaxation improves upon the aforementioned approaches.

\subsection{SDP relaxation}


It is standard to relax the non-convex constraints \eqref{QCQP_tensor_lift_d} as the SDP constraint
\begin{align}
	\Xbf\succeq \Ybf\Ybf^{\!\top}
	\quad\Leftrightarrow\quad
	\begin{bmatrix}
		\!\!\!\Xbf & \!\!\!\Ybf\\
		\Ybf^{\!\top} & \Ibf_{\! m}\\
	\end{bmatrix}\succeq 0,\label{sdp_cons}
\end{align}
leading to an SDP relaxation of \eqref{QCQP_vector_a} -- \eqref{QCQP_vector_c}. Both SDP and parabolic relaxations share CBP which makes them compatible with objective penalization and Algorithm \ref{alg:1}. 

\vspace{2mm}

\begin{itemize}
	\item[a)] {\it Quality:} In terms of quality, the parabolic relaxation is dominated by SDP relaxation, i.e.,
	\begin{align}
		\!\!\!\!\Xbf\!\succeq\!\Ybf\Ybf\T \!\Rightarrow\! 
		\left\{\begin{matrix}
			(\wbf_i\!+\!\wbf_j)^{\!\top}\Xbf(\wbf_i\!+\!\wbf_j)\geq \|(\wbf_i\!+\!\wbf_j)^{\!\top}\Ybf\|_2^2,\quad \forall i,j\!\in\!\Ncal,\\
			(\wbf_i\!-\!\wbf_j)^{\!\top}\Xbf(\wbf_i\!-\!\wbf_j)\geq \|(\wbf_i\!-\!\wbf_j)^{\!\top}\Ybf\|_2^2,\quad \forall i,j\!\in\!\Ncal.
		\end{matrix}\right.
	\end{align}
	In section \ref{coord}, it is shown that an appropriate choice of $\{\wbf_i\}_{i\in\Ncal}$ results in a relaxation as strong as SDP. However, finding such basis can be as costly as solving SDP relaxation, and therefore, it may be justifiable for cases where a sequence of problems should be solved (e.g., through Algorithm \ref{alg:1} or branch-and-bound search). 
	
	\vspace{2mm}
	\item[b)] {\it Complexity:}
	In the presence of sparsity, large-scale SDP constraints of the form \eqref{sdp_cons} are decomposable into smaller ones through graph-theoretic analysis \cite{fukuda2001exploiting,nakata2003exploiting}. Despite this possibility, a large body of literature recommends the use of simpler relaxations such as  SOCP, which allow the omission of those elements in $\Xbf$ that do not participate in the objective \eqref{QCQP_tensor_lift_a} and constraints \eqref{QCQP_tensor_lift_b} -- \eqref{QCQP_tensor_lift_c} (see \cite{majumdar2019recent}). For sparse problems, the parabolic relaxation enjoys this advantage as well, which makes it a viable alternative to SDP relaxation for a broad class of real-world problems. For dense problems, the complexity of imposing \eqref{sdp_cons} versus $n^2$ convex quadratic constraints of the form \eqref{QCQP_tensor_para_d} -- \eqref{QCQP_tensor_para_e} is highly-dependent on the choice of numerical algorithm and the type of application. Such comparison requires and in-depth and case-by-case analysis of computational complexity which is beyond the scope of this paper.
	
\end{itemize}

\begin{figure}[p!]	
\begin{adjustbox}{angle=90}
	\scalebox{0.950}{
		\begin{tabular}{ 
				|lccccccccccccccccc| 
			}
			\hline 
			&
			\cellcolor{red!25}\small{The matrix} &  &
			\cellcolor{red!25}\small{$(n-1)\times (n-1)$} &  &
			&  &
			\cellcolor{red!25}\small{$2\times 2$ principal} &  &
			\cellcolor{red!25}\!\!\!\small{Diagonal elements}\!\!\! &  &
			& &
			& &
			& &
			\\ 
			&
			\cellcolor{red!25}\small{$\Xbf-\Ybf\Ybf\T$} &  &
			\cellcolor{red!25}\small{principal submatrices} &  &
			&  &
			\cellcolor{red!25}\small{submatrices of} &  &
			\cellcolor{red!25}\small{\!of $\Xbf\!\!-\!\Ybf\!\Ybf\T$\!\!} &  &
			& &
			& &
			& &
			\\
			&
			\cellcolor{red!25}\small{is PSD.} &  &
			\cellcolor{red!25}\small{\!of $\Xbf\!\!-\!\Ybf\!\Ybf\T$\!\! are PSD.\!} &  &
			&  &
			\cellcolor{red!25}\small{\!\!$\Xbf\!\!-\!\Ybf\Ybf\T\!\!$ are PSD.\!\!} &  &
			\cellcolor{red!25}\small{are positive.} &  &
			& &
			& &
			& &
			\\
			\hline 
			\hline 
			\!Number:\!\!\!\! &
			\cellcolor{red!25}\!\!$\Big(\begin{array}{c} n \\ n \end{array}\Big)$\!\! &  &
			\cellcolor{red!25}\!\!$\Big(\begin{array}{c} n \\ n-1 \end{array}\Big)$\!\! &  &
			&  &
			\cellcolor{red!25}\!\!$\Big(\begin{array}{c} n \\ 2 \end{array}\Big)$\!\! &  &
			\cellcolor{red!25}\!\!$\Big(\begin{array}{c} n \\ 1 \end{array}\Big)$\!\! &  &
			& &
			& &
			& &
			\\
			\!Type:\!\!\rule{0pt}{15pt}
			&
			\cellcolor{red!25}\!\!\small{$(n\!+\!m)\!\times\!(n\!+\!m)$}\!\! & &
			\cellcolor{red!25}\!\!\!\small{$(n\!-\!1\!+\!m)\!\times\!(n\!-\!1\!+\!m)$}\!\!\! &  &
			& &
			\cellcolor{red!25}\!\!\small{$(2\!+\!m)\!\times\!(2\!+\!m)$}\!\! &  &\
			\cellcolor{red!25}Quadratic &  &
			& &
			& &
			& & 
			\\
			\!Set:\!\!\rule{0pt}{20pt}
			&
			\cellcolor{red!25}$\Scal^{\;(n+m)}_{n,m}$ & \!\!$\subseteq$\!\! & 
			\cellcolor{red!25}$\Scal^{\;(n+m-1)}_{n,m}$ & \!\!$\subseteq$\!\! &
			\!\!\dots\!\! & \!\!$\subseteq$\!\!  &
			\cellcolor{red!25}$\Scal^{\;(2+m)}_{n,m}$ & \!\!$\subseteq$\!\! &
			\cellcolor{red!25}$\Scal^{\;(1+m)}_{n,m}$ &  &
			& &
			& &
			& &
			\\
			\rule{0pt}{12pt}
			&
			& & 
			& &
			& & \multicolumn{1}{r}{\begin{turn}{-45}$\subseteq$\end{turn} }
			& & \multicolumn{1}{l}{\begin{turn}{45}\!\!\!\!$\subseteq$\end{turn} }
			& &
			& &
			& &
			& &
			\\
			\rule{0pt}{10pt}
			&
			\begin{turn}{90}$=$\end{turn} &  & 
			\begin{turn}{90}$\subseteq$\end{turn} &  &
			& &
			\begin{turn}{90}$\subseteq$\end{turn} &  
			\cellcolor{green!25}$\Pcal_{n,m}$ &
			\begin{turn}{90}$\subseteq$\end{turn} &  &
			& &
			& &
			& &
			\\
			\rule{0pt}{12pt}
			&
			& & 
			& &
			& & \multicolumn{1}{r}{\begin{turn}{45}$\subseteq$\end{turn} }
			& & \multicolumn{1}{l}{\!\!\!\begin{turn}{135}$\not\subseteq$\end{turn}  }
			& & 
			& &
			& &
			& &
			\\
			\!Set:\!\!\rule{0pt}{15pt}
			&
			\cellcolor{blue!25}$\bar{\Scal}^{\;(n+m)}_{n,m}$ & \!\!$\subseteq$\!\! & 
			\cellcolor{blue!25}$\bar{\Scal}^{\;(n-1+m)}_{n,m}$ & \!\!$\subseteq$\!\! &
			\!\!\dots\!\! & \!\!$\subseteq$\!\!  &
			\cellcolor{blue!25}$\bar{\Scal}^{\;(2+m)}_{n,m}$ & \!\!$\subseteq$\!\! &
			\cellcolor{blue!25}$\bar{\Scal}^{\;(1+m)}_{n,m}$ & \!\!$\subseteq$\!\! &
			\!\!\dots\!\! & \!\!$\subseteq$\!\!  &
			\cellcolor{blue!25}$\bar{\Scal}^{\;(3)}_{n,m}$ & \!\!$\subseteq$\!\! &
			\cellcolor{blue!25}$\bar{\Scal}^{\;(2)}_{n,m}$ & \!\!$\subseteq$\!\! &
			\cellcolor{blue!25}$\bar{\Scal}^{\;(1)}_{n,m}$ \\
			\!Type:\!\!\rule{0pt}{15pt}
			&
			\cellcolor{blue!25}\!\!\small{$(n\!+\!m)\!\times\!(n\!+\!m)$}\!\! & &
			\cellcolor{blue!25}\!\!\!\small{$(n\!-\!1\!+\!m)\!\times\!(n\!-\!1\!+\!m)$}\!\!\! &  &
			& &
			\cellcolor{blue!25}\!\!\small{$(2\!+\!m)\!\times\!(2\!+\!m)$}\!\! &  &\
			\cellcolor{blue!25}\!\!\small{$(1\!+\!m)\!\times\!(1\!+\!m)$}\!\! &  &
			& &
			\cellcolor{blue!25}\!\!\small{$3\!\times\! 3$}\!\! & &
			\cellcolor{blue!25}\!\!\small{SOCP}\!\! & & 
			\cellcolor{blue!25}\!\!\small{Linear}\!\! \\
			\!Number:\!\!\!\!
			&
			\cellcolor{blue!25}\!\!$\Big(\begin{array}{c} n+m \\ n+m \end{array}\Big)$\!\! &  &
			\cellcolor{blue!25}\!\!$\Big(\begin{array}{c} n+m \\ n-1+m \end{array}\Big)$\!\! &  &
			&  &
			\cellcolor{blue!25}\!\!$\Big(\begin{array}{c} n+m \\ 2+m \end{array}\Big)$\!\! &  &
			\cellcolor{blue!25}\!\!$\Big(\begin{array}{c} n+m \\ 1+m \end{array}\Big)$\!\! &  &
			& &
			\cellcolor{blue!25}\!\!$\Big(\begin{array}{c} n+m \\ 3 \end{array}\Big)$\!\!\! &  &
			\cellcolor{blue!25}\!\!$\Big(\begin{array}{c} n+m \\ 2 \end{array}\Big)$\!\!\! &  &
			\cellcolor{blue!25}\!\!$\Big(\begin{array}{c} n+m \\ 1 \end{array}\Big)$\!\!\!\\
			\hline
			\hline
			&
			\cellcolor{blue!25}\small{The matrix} &  &
			\cellcolor{blue!25}\!\!\small{$(n\!-\!1\!+\!m)\!\times\!(n\!-\!1\!+\!m)$}\!\! &  &
			&  &
			\cellcolor{blue!25}\!\!\small{$(2\!+\!m)\!\times\!(2\!+\!m)$}\!\! &  &
			\cellcolor{blue!25}\!\!\small{$(1\!+\!m)\!\times\!(1\!+\!m)$}\!\! &  &
			& &
			\cellcolor{blue!25}\!\!\small{$3\!\times\!3$}\!\!
			& &
			\cellcolor{blue!25}\!\!\small{$2\!\times\!2$}\!\!
			& &
			\cellcolor{blue!25}\!\!\small{Diagonal}\!\!
			\\ 
			&
			\cellcolor{blue!25}\!\!\!\small{
				$\begin{bmatrix}
					\Xbf\Tp & \Ybf\\
					\Ybf\T & \Ibf\\
				\end{bmatrix}$}\small{is PSD.}\!\!\! &  &
			\cellcolor{blue!25}\small{principal} &  &
			&  &
			\!\!\!\cellcolor{blue!25}\small{principal}\!\!\! &  &
			\!\!\!\cellcolor{blue!25}\small{principal}\!\!\! &  &
			& &
			\!\!\!\cellcolor{blue!25}\small{principal}\!\!\!
			& &
			\!\!\!\cellcolor{blue!25}\small{principal}\!\!\!
			& &
			\!\!\!\cellcolor{blue!25}\small{elements of}\!\!\!
			\\
			&
			\cellcolor{blue!25} &  &
			\cellcolor{blue!25}\small{submatrices of} &  &
			&  &
			\!\!\!\cellcolor{blue!25}\small{submatrices of}\!\!\! &  &
			\!\!\!\cellcolor{blue!25}\small{submatrices of}\!\!\! &  &
			& &
			\!\!\!\cellcolor{blue!25}\small{submatrices of}\!\!\!
			& &
			\!\!\!\cellcolor{blue!25}\small{submatrices of}\!\!\!
			& &
			\cellcolor{blue!25}\!\!\!\small{
				$\begin{bmatrix}
					\Xbf\Tp & \Ybf\\
					\Ybf\T & \Ibf\\
				\end{bmatrix}$}
			\\
			&
			\cellcolor{blue!25} & &
			\!\!\!\cellcolor{blue!25}\small{\!\!\! 
				$\begin{bmatrix}
					\Xbf\Tp & \Ybf\\
					\Ybf\T & \Ibf\\
				\end{bmatrix}$\!\! are PSD.\!\!}\!\!\! &  &
			& &
			\cellcolor{blue!25}\small{\!\!
				$\begin{bmatrix}
					\Xbf\Tp & \Ybf\\
					\Ybf\T & \Ibf\\
				\end{bmatrix}$
				are PSD.\!\!} &  &
			\!\!\!\cellcolor{blue!25}\small{\!\!\! 
				$\begin{bmatrix}
					\Xbf\Tp & \Ybf\\
					\Ybf\T & \Ibf\\
				\end{bmatrix}$\!\! are PSD.\!\!}\!\!\!  &  &
			& &
			\!\!\!\cellcolor{blue!25}\small{\!\!\! 
				$\begin{bmatrix}
					\Xbf\Tp & \Ybf\\
					\Ybf\T & \Ibf\\
				\end{bmatrix}$\!\! are PSD.\!\!}\! 
			& &
			\!\!\!\cellcolor{blue!25}\small{\!\!\! 
				$\begin{bmatrix}
					\Xbf\Tp & \Ybf\\
					\Ybf\T & \Ibf\\
				\end{bmatrix}$\!\! are PSD.\!\!}\! 
			& &
			\!\!\cellcolor{blue!25}\small{are positive.}\!\!\!
			\\
			\hline
			\hline
			& \cellcolor{blue!25}\small{SDP}
			& & 
			& &
			& & \!\!\!\!\cellcolor{blue!25}\small{SOCP relaxation in}\!\!\!\!
			& & \cellcolor{blue!25}\small{SOCP relaxation}
			& & 
			& &
			& &  \cellcolor{blue!25}\small{SOCP relaxation}
			& &\\
			& \cellcolor{blue!25}\small{relaxation}
			& & 
			& &
			& & \cellcolor{blue!25}\small{homogeneous case}
			& & \cellcolor{blue!25}\small{ for $m=1$}
			& & 
			& &
			& &  \cellcolor{blue!25}\small{in general case}
			& &
			\\
			\hline
	\end{tabular}}
 \end{adjustbox}
	\caption{
		Comparison between the two convex relaxation hierarchies in Definition \ref{def_hie} with basic parabolic relaxation in terms of quality, the number and the type of conic inequalities in the worst case (dense matrices).}
	\vspace{-5mm}
	\label{fig_compare}
\end{figure}

\begin{figure*}[t]
	\centering
	\includegraphics[height=0.60\textwidth]{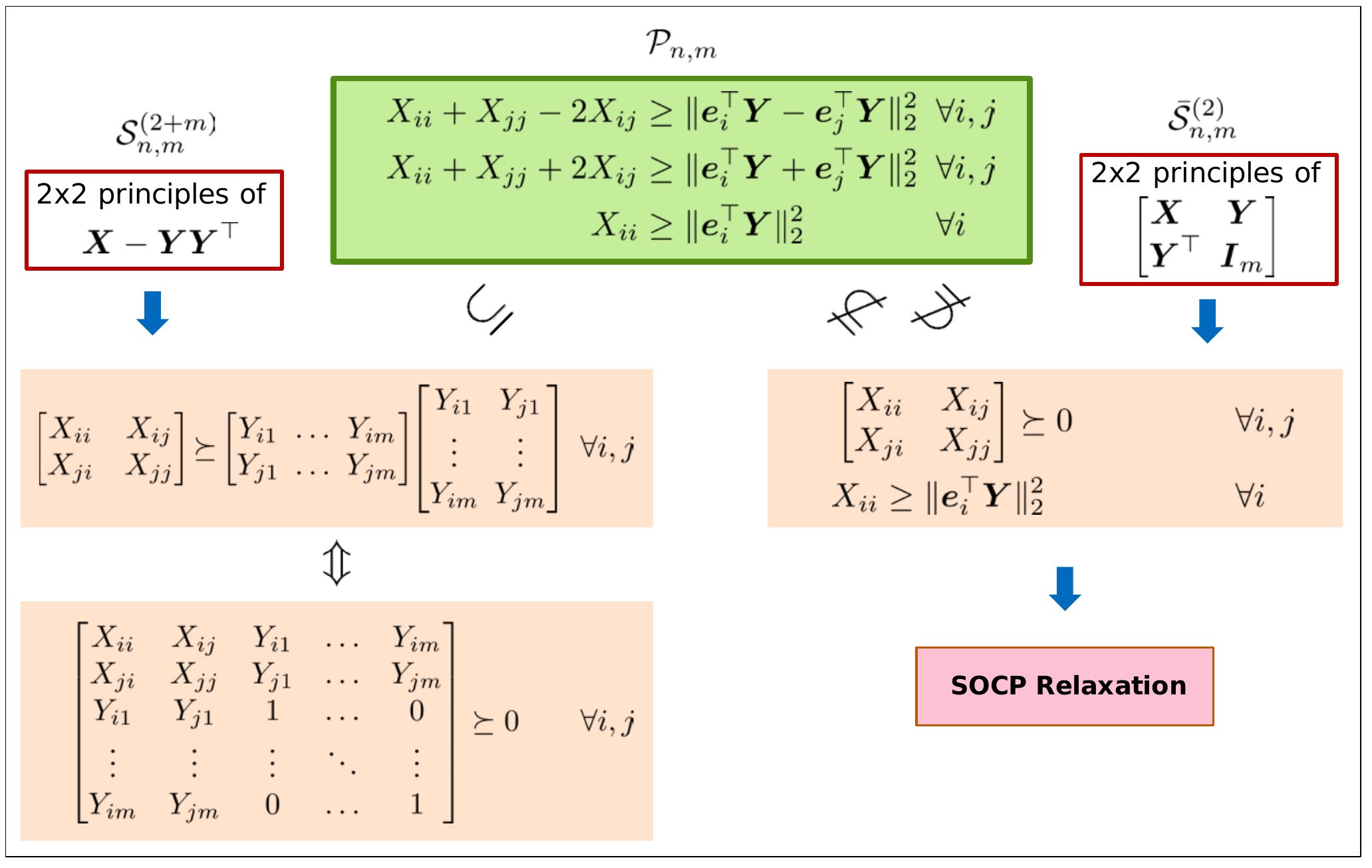}
	\caption{
		Comparison between the parabolic and two other relaxations of $\Xbf=\Ybf\Ybf^{\top}$.
	}
	\label{fig_comp}
\end{figure*}

\vspace{2mm}
\subsection{Outer-approximations of SDP relaxation}
To alleviate the cost of solving large-scale semidefinite programs, it is common-practice to substitute the constraint \eqref{sdp_cons} with a set of conic inequalities of the form:
\begin{align}
	\Ob^{\top}_{k}
	\begin{bmatrix}
		\!\!\!\Xbf & \!\!\!\Ybf\\
		\Ybf^{\!\top} & \Ibf_m\\
	\end{bmatrix}
	\Ob_{ k}
	\succeq 0, \qquad\qquad \forall k\in \Ocal,\label{O}
\end{align}
where $\{\Ob_{k}\in\Rbb^{(n+m)\times n_k}\}_{k\in\Vcal}$ are arbitrary matrices, each with $n+m$ rows \cite{ahmadi2017dsos}. Our approach is closely related to this method, since the parabolic relaxation inequalities \eqref{QCQP_tensor_para_d} -- \eqref{QCQP_tensor_para_e} can be reformulated as:
\begin{align}
	\obf^{\top}_{k}
	(\Xbf-\Ybf\Ybf^{\!\top})
	\obf_{k}
	\succeq 0, \qquad\qquad \forall k\in \Ocal, \label{WXY}
\end{align}
where
\begin{align}
	\{\obf_k\}_{k\in\Ocal}\triangleq
	\big\{\wbf_i+\wbf_j\,|\,i,j\in\{1,\ldots,n\}\big\}\cup\big\{\wbf_i-\wbf_j\,|\,i,j\in\{1,\ldots,n\}\big\}.
\end{align}
We are motivated by the observation that the resulting inequalities in this case 
are convex quadratic. 

To further assess the proposed parabolic relaxation, the following definition introduces two hierarchies of convex relaxations. 
\begin{definition}\label{def_hie}
	Define $\Sbb^{+}_{o;\,k}$ as the set of symmetric $o\times o$ matrices whose every $k\times k$ principle submatrix is positive semidefinite. Define the hierarchy
	$\bar{\Scal}^{(n+m)}_{n,m} \subseteq \bar{\Scal}^{(n-1+m)}_{n,m}$ 
	$\subseteq \ldots \subseteq \bar{\Scal}^{(1)}_{n,m}$
	as:
	\begin{align}
		\bar{\Scal}^{(k)}_{n,m}\triangleq\left\{(\Ybf,\Xbf)\in\Rbb^{n\times m}\times\Sbb_n\;\Big|\;
		\sm{
			\begin{bmatrix}
				\Xbf\Tp & \Ybf\Tp \\
				\Ybf\T & \Ibf_{\! m}\\
		\end{bmatrix}}{24}\!
		\in\Sbb^{+}_{(n+m);\,k}\right\}.
	\end{align}
	Define, also
	$\Scal^{(n+m)}_{n,m} \subseteq \Scal^{(n-1+m)}_{n,m} \subseteq \ldots \subseteq \Scal^{(1+m)}_{n,m}$
	as:
	\begin{align}
		\Scal^{(k)}_{n,m}\triangleq\{(\Ybf,\Xbf)\in\Rbb^{n\times m}\times\Sbb_n\;|\;\Xbf-\Ybf\Ybf\T\in\Sbb^{+}_{n;\,k}\}.
	\end{align}
	Lastly, define $\Pcal_{n,m}$ as the set of all pairs $(\Ybf,\Xbf)\in\Rbb^{n\times m}\times\Sbb_n$ that satisfy \eqref{QCQP_tensor_para_stan_d} -- \eqref{QCQP_tensor_para_stan_e}. 
\end{definition}
According to Schur complement we have $\bar{\Scal}^{(n+m)}_{n,m}=\Scal^{(n+m)}_{n,m}$; however, the two hierarchies are different in terms of strength, the number of conic inequalities, and the complexity of conic inequalities. Figure \eqref{fig_compare} illustrates a general comparison between members of $\{\Scal^{(k)}_{n,m}\}_{k\in\{1+m,\ldots,n+m\}}$, $\{\bar{\Scal}^{(k)}_{n,m}\}_{k\in\{1,\ldots,n+m\}}$, and $\Pcal_{n,m}$. In the remainder of this section, we will compare the parabolic relaxation with these three {different} outer-approximations.

\vspace{3mm}
\noindent{\bf Dual DD programming:} which is a restriction in the dual domain leads to the enforcement of:
\begin{align}
	\begin{bmatrix}
		\Abf_0 & \Bbf_0 \\
		\Bbf_0\T & -\Tbf+m^{-1}c_0\Ibf_{\! m}\\
	\end{bmatrix}+
	\sum_{k\in\Ecal\cup\Ical}{\tau_k
		\begin{bmatrix}
			\Abf_k & \Bbf_k \\
			\Bbf_k\T & m^{-1}c_k\Ibf_{\! m}\\
	\end{bmatrix}}\in\Dbb_{n+m},
\end{align}
in lieu of \eqref{QCQP_tensor_para_dual_b}. This is equivalent to a relaxation of the form \eqref{O} in primal domain corresponding to the choice:
\begin{align}
	\!\!\!\!\{\obf_k\}_{k\in\Ocal}\!=\! \big\{\ebf_i\!+\!\ebf_j\,|\,i,j\in\{1,\ldots,n\!+\!m\}\big\}\!\cup\!\big\{\ebf_i\!-\!\ebf_j\,|\,i,j\in\{1,\ldots,n\!+\!m\}\big\}.\!\!\!\!
\end{align}

Here is how this method compares with the parabolic relaxation approach:
\begin{itemize}
	\item[a)] {\it Quality:} By investigating the dual parabolic relaxation \eqref{QCQP_tensor_para_dual_a} -- \eqref{QCQP_tensor_para_dual_d}, it can be easily observed that this approach is dominated by the parabolic relaxation approach for $m>0$, and it is effectively equivalent to the parabolic relaxation for the homogeneous case.
	
	\vspace{2mm}
	\item[b)] {\it Complexity:} The dual DD programming enjoys lower computational complexity due to its reliance on linear constraints.
\end{itemize}

\vspace{2mm}
\noindent{\bf First variant of SOCP relaxation (Dual SDD programming):}
The first variant of SOCP relaxation is $\bar{\Scal}^{(2)}_{n,m}$, which in the worst case involves 
$\small\Big(\begin{array}{c} n+m \\ 2 \end{array}\Big)$ conic inequalities, compared to $n^2$ convex quadratic inequalities for the parabolic relaxation.  Here is how this method compares with the parabolic relaxation approach:
\begin{itemize}
	\item[a)] {\it Quality:} Unless in the homogeneous case, this relaxation is not stronger than the parabolic relaxation in general, which is demonstrated by Example \ref{exm_1} and benchmark QCQP problems in \cite[Section 3.1]{parabolic_part2}. 
	
	\vspace{2mm}
	\item[b)] {\it Complexity:} This variant of SOCP relaxation expands the sparsity pattern of the problem by adding an additional clique of size $m\times m$. This can potentially increase complexity.
\end{itemize}

\vspace{3mm}
\noindent{\bf Second variant of SOCP relaxation:}
This variant of SOCP relaxation involves vectorization of $\Ybf$ by considering a variable:
\begin{align}
	\ybf\triangleq[(\Ybf\!\ebf_1)\T,\; (\Ybf\!\ebf_2)\T,\; \ldots,\; (\Ybf\!\ebf_m)\T]^{\top}\in\Rbb^{mn},
\end{align}
as well as a lifted matrix variable representing $\ybf\ybf^{\top}\in\Sbb_{mn}$. Here is how this method compares with the parabolic relaxation approach:
\begin{itemize}
	\item[a)] {\it Quality:} A direct comparison between this variant of SOCP and the ones introduced in Definition \ref{def_hie} is not straightforward because the lifted variable belongs to a higher-dimensional space $\Sbb_{mn}$.
	
	\vspace{2mm}
	\item[b)] {\it Complexity:} This SOCP variant can negatively affect the sparsity pattern of the problem by requiring a very large number of scalar variables accounting for the elements of $\ybf\ybf^{\top}\in\Sbb_{mn}$.
\end{itemize}

\vspace{3mm}
\noindent{\bf $2\times2$ principals of $\Xbf-\Ybf\Ybf^{\top}$:}
This relaxation is denoted by $\Scal^{(2+m)}_{n,m}$ which is studied in \cite{madanipolynomial}. In \cite{madanipolynomial}, the authors show that $\Scal^{(2)}_{n,m}$ satisfies CBP for $m=1$.  Here is how this method compares with the parabolic relaxation approach:
\paragraph{Quality:}
It can be easily observed that $\Scal^{(2+m)}_{n,m}$ is always stronger than the parabolic, i.e., $\Scal^{(2+m)}_{n,m}\subseteq\Pcal_{n,m}$.

\paragraph{Complexity:} The main issue with this approach is that the resulting conic inequalities are of the type $(m+2)\times(m+2)$ positive semidefinite:
\begin{align}
	\begin{bmatrix}
		X_{ii} & X_{ij} & \ebf_i\T \Ybf\\
		X_{ij} & X_{jj} & \ebf_j\T \Ybf\\
		\Ybf\T\ebf_i & \Ybf\T\ebf_j & \Ibf_m
	\end{bmatrix}\succeq 0,
\end{align}
which are not SOCP except in the homogeneous case. Such constraints can be computationally demanding for $m>0$ and they are not compatible with conic quadratic branch-and-bound solvers \cite{gurobi}. Figure \ref{fig_comp} provides an illustration between $\Pcal_{n,m}$, $\Scal^{(2+m)}_{n,m}$, and $\bar{\Scal}^{(2)}_{n,m}$.

\subsection{Inner-approximations of SDP relaxation}
An alternative method is the employment of lower-complexity inner-approximations of SDP cone in the form of:
\begin{align}\!\!\!
	\begin{bmatrix}
		\!\!\!\Xbf & \!\!\!\Ybf\\
		\Ybf^{\!\top} & \Ibf_{\! m}\\
	\end{bmatrix}
	=\sum_{i\in \Ucal}\;
	\Ubf_{\! i}\Tp\!
	\Zbf_i\Tp
	\Ubf_{\! i}\T,
	\qquad\wedge\qquad
	\Zbf_i\Tp \succeq 0, \qquad
	\forall i\in \Ucal,
\end{align}
where 
$\{\Zbf_{\!i}\in\Sbb^+_{n_i}\}_{i\in\Ucal}$ are auxiliary positive semidefinite variables and\linebreak $\{\Ubf_{\!i}\in\Rbb^{(n+m)\times n_i}\}_{i\in\Ucal}$ are arbitrary matrices, each with $n+m$ rows \cite{ahmadi2017dsos}. The two most common inner-approximations are the DD and SDD corresponding to the following choices:
\begin{subequations}
	\begin{align}
		&\{\Ubf_{\!i}\}=
		\big\{\ebf_i\!-\!\ebf_j\;|\;\forall i\neq j\in\{1,\ldots,n+m\}\big\}\nonumber\\ 
		&\phantom{\{\Ubf_{\!i}\}}\;\cup
		\big\{\ebf_i\!+\!\ebf_j\;|\;\forall i\neq j\in \{1,\ldots,n+m\}\big\},\\
		&\{\Ubf_{\!i}\}=
		\big\{[\ebf_i\;\;\ebf_j]\;|\;\forall i\neq j\in \{1,\ldots,n+m\}\big\},
	\end{align}
\end{subequations}
respectively. Despite the broad application of DD and SDD approximations in control theory \cite{ahmadi2017dsos}, their main drawback for solving QCQP is that they are neither relaxations nor restrictions of the original non-convex feasible set $\Fcal$. Hence, in certain applications the resulting inner-approximation can be empty or exclude meaningful solutions. This is illustrated in the next section.


\section{Illustrative Examples}

In this section we present two examples illustrating the strength of the alternative relaxations and approximations as well as the sequential penalized parabolic relaxations.


\subsection{Example 1}\label{exm_1}

The following illustrative example demonstrates a comparison between the convex feasible set generated via the parabolic relaxation and other methods that are discussed in the previous sections on a simple problem. Consider an optimization problem with variables $x_1$ and $x_2$ whose feasible set $\Fcal\subset\Rbb^2$ is the intersection of two non-convex quadratic constraints:
\begin{align}
	\Fcal\triangleq\{\xbf\in\Rbb^2\mid& \; x_1^2-18x_1x_2+0.1x_2^2+2x_1-3x_2\leq 1,\nonumber\\
	& \; x_2^2+24x_1x_2-x_1+3x_2+0.1\leq 0\}.\label{eq_exmp1}
\end{align}
Following is the lifted formulation of $\Fcal$:
\begin{align}
	\Fcal^{\mathrm{lifted}}\triangleq\{(\xbf,\Xbf)\!\in\!\Rbb^2\!\times\!\Sbb_2\mid
	& \; X_{11} - 18 X_{12} + 0.1 X_{22}+2x_1 - 3x_2 \leq 1,\nonumber\\
	& \; X_{22} + 24 X_{12} - x_1 + 3x_2 + 0.1 \leq 0\}.
\end{align}
The relationship between $\xbf$ and $\Xbf$ can be imposed implicitly by restricting $(\xbf,\Xbf)$ to the intersection of $\Fcal^{\mathrm{lifted}}$ with $\bar{\Scal}^{(3)}_{2,1}$, $\bar{\Scal}^{(2)}_{2,1}$, $\Pcal_{2,1}$ or one of the the following sets:
\begin{subequations}
	\begin{align}
		&\hspace{-0.25cm}\Scal^{\mathrm{DD}}\!\triangleq\!\left\{(\xbf,\Xbf)\in\Rbb^2\!\times\!\Sbb_2\;|\;
		X_{11} \geq |X_{12}| + |x_1|,\quad\right.\nonumber\\
		&\hspace{-0.25cm} \qquad\qquad\quad\;\qquad\qquad\qquad\left. X_{22} \geq |X_{21}| + |x_2|,\quad
		1 \geq |x_1| + |x_2|\right\},\\
		&\hspace{-0.25cm}\Scal^{\mathrm{SDD}}\!\triangleq\!\Bigg\{\!(\xbf,\Xbf)\in\Rbb^2\!\times\!\Sbb_2\!\;\Bigg|\nonumber\\
		&\hspace{-0.25cm}{\small\begin{bmatrix}
				X_{11} & X_{12} & x_{1}\\
				X_{21} & X_{22} & x_{2}\\
				x_{1} & x_{2} & 1\\
		\end{bmatrix}}\!\!=\!\!
		{\small\begin{bmatrix}
				H^{(1)}_{11}\!+\!H^{(2)}_{11} & H^{(1)}_{12} & H^{(2)}_{12}\\
				H^{(1)}_{21} & H^{(1)}_{22}\!+\!H^{(3)}_{11} & H^{(3)}_{12}\\
				H^{(2)}_{21} & H^{(3)}_{21} & H^{(3)}_{22}\!+\!H^{(2)}_{22}\\
		\end{bmatrix}},\;\Hbf^{(1)}\!,\Hbf^{(2)}\!,\Hbf^{(3)}\!\succeq\!0\Bigg\},\!\!
	\end{align}
\end{subequations}
leading to SDP, SOCP (both variants), and parabolic relaxations, as well as the DD and SDD approximations, respectively. As demonstrated in Figures \eqref{fig_para}.A and \eqref{fig_para}.B, DD and SDD are neither inner- nor outer-approximations of the original feasible set. On the other hand, SOCP relaxation in \eqref{fig_para}.C offers a loose relaxation compared to the parabolic relaxation in \eqref{fig_para}.D.

\begin{figure*}[t]
	\centering
	\includegraphics[height=0.24\textwidth]{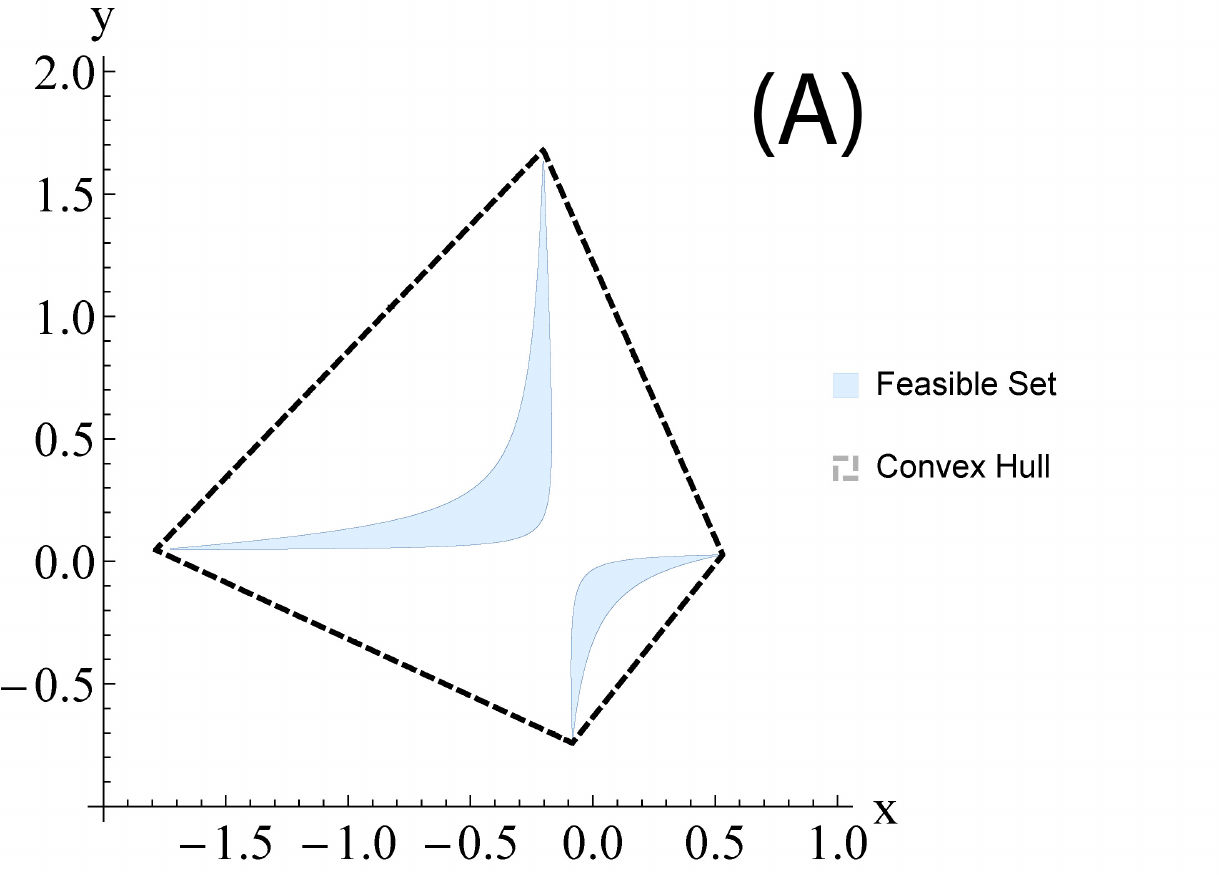}
	\includegraphics[height=0.24\textwidth]{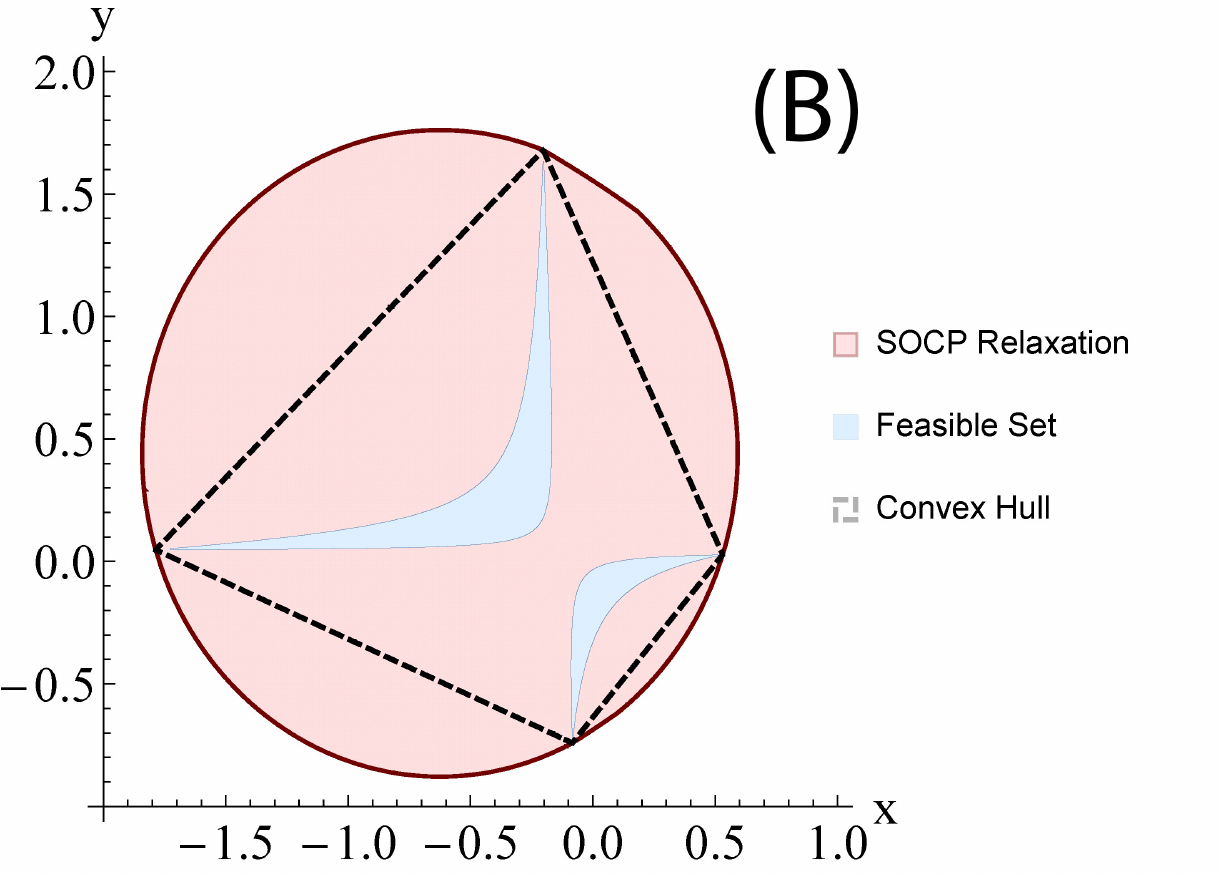}
	\includegraphics[height=0.24\textwidth]{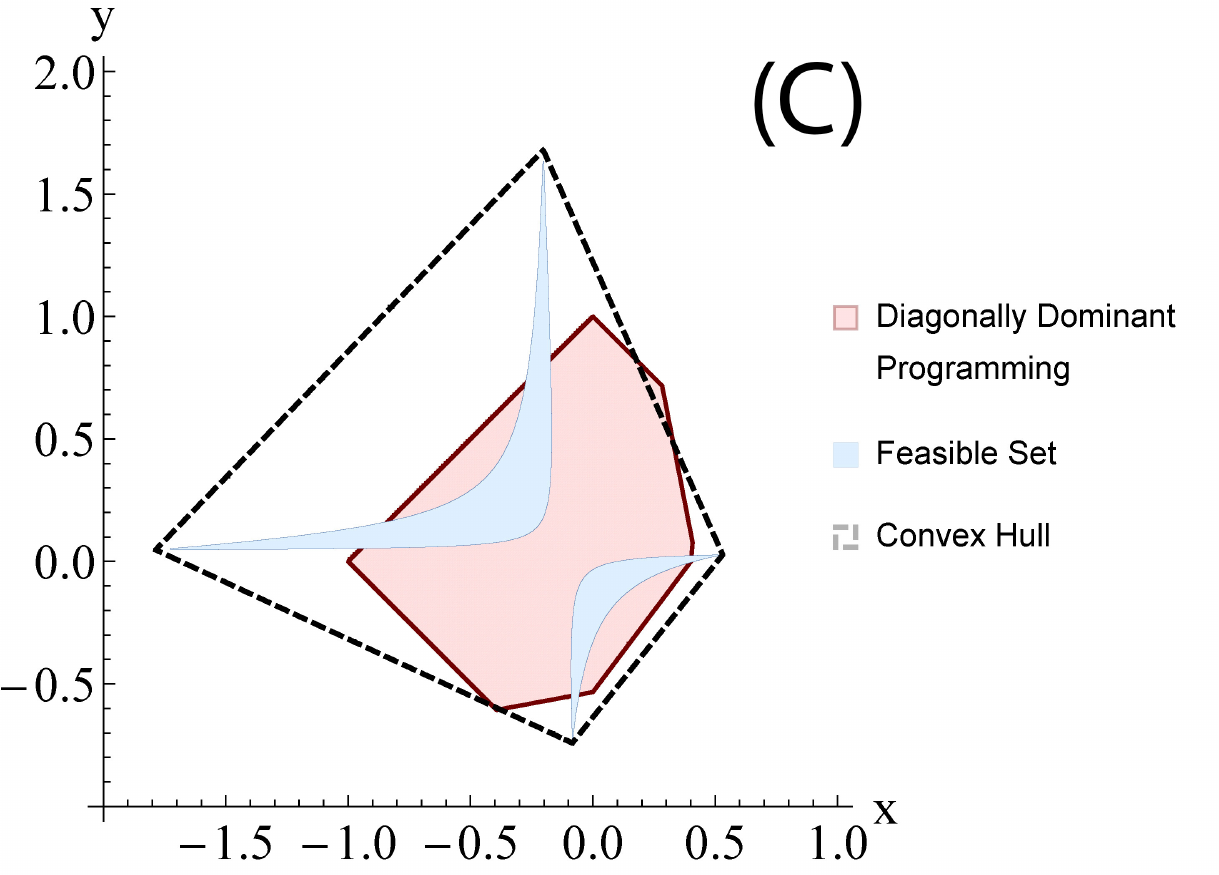}\\
	\includegraphics[height=0.24\textwidth]{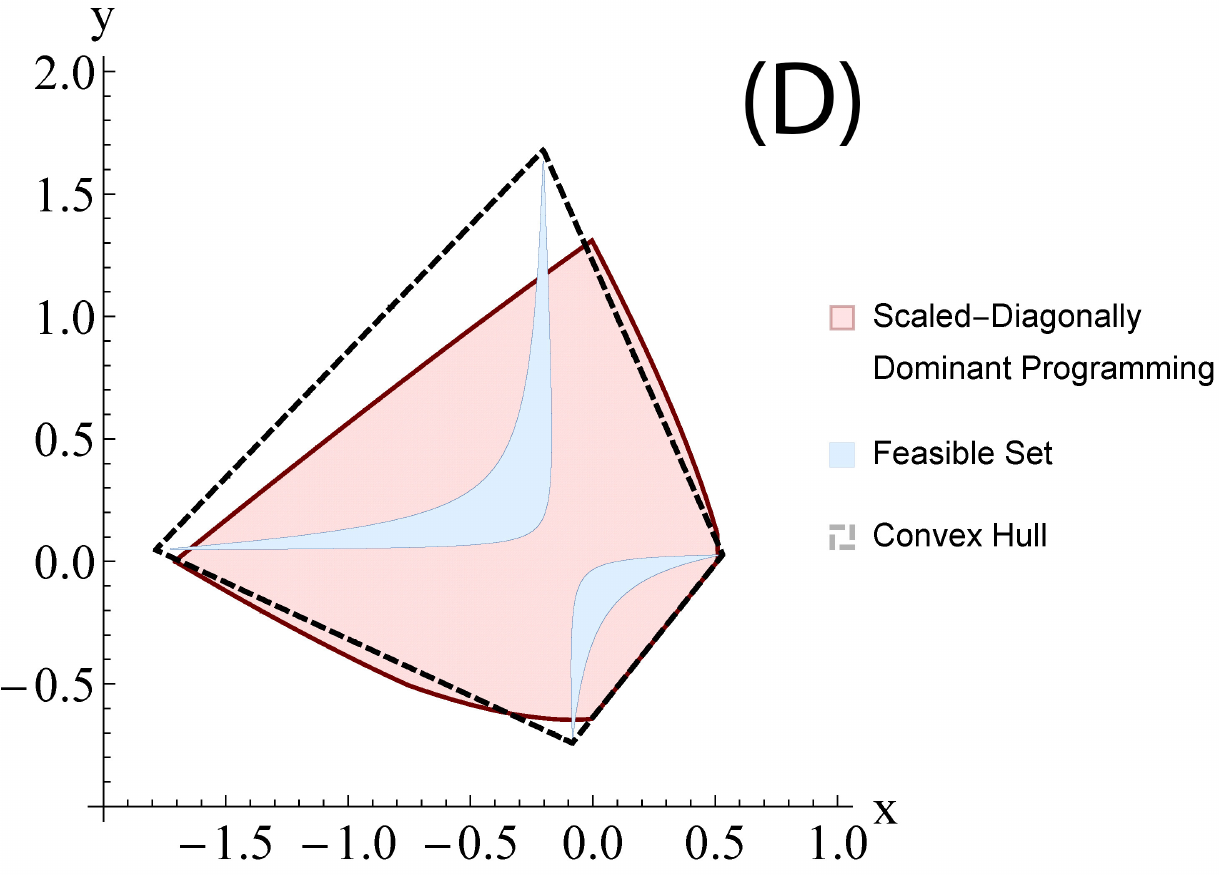}
	\includegraphics[height=0.24\textwidth]{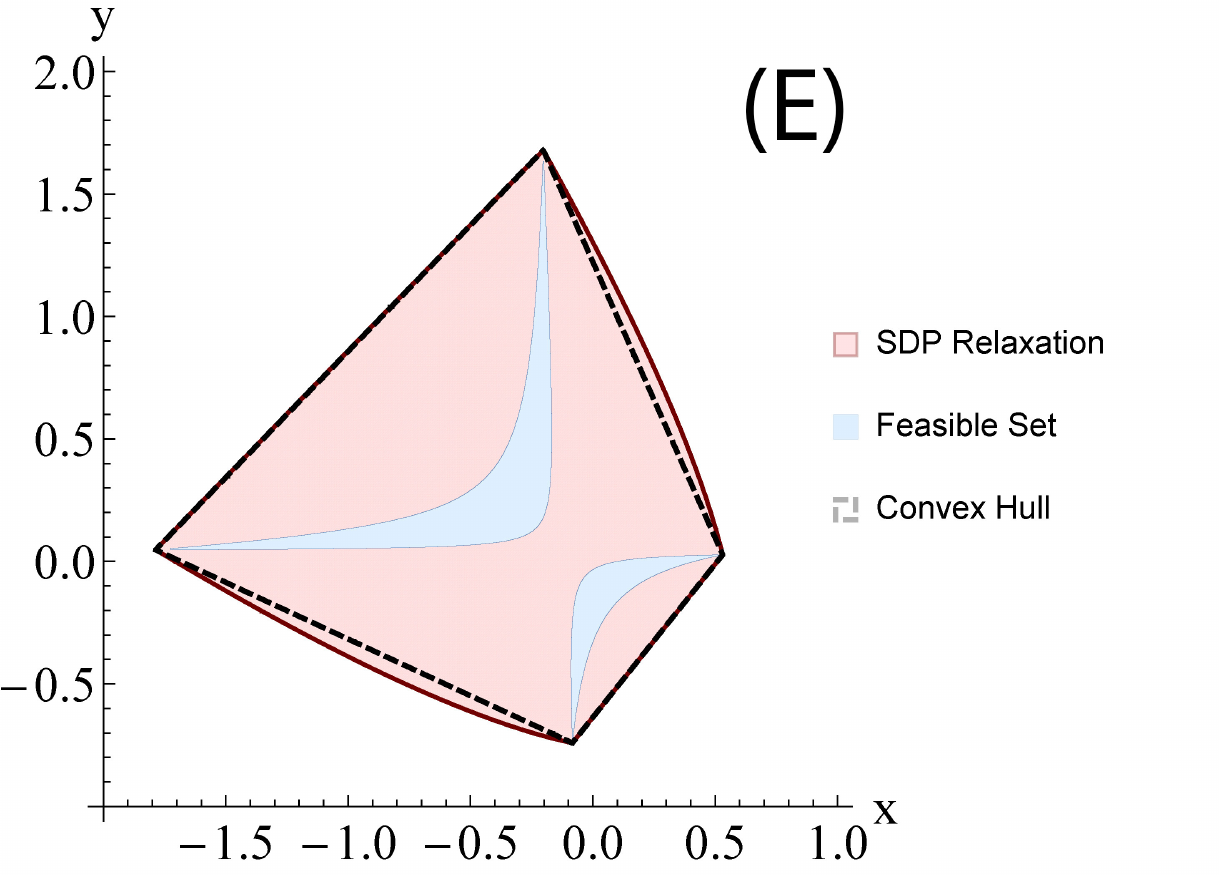}
	\includegraphics[height=0.24\textwidth]{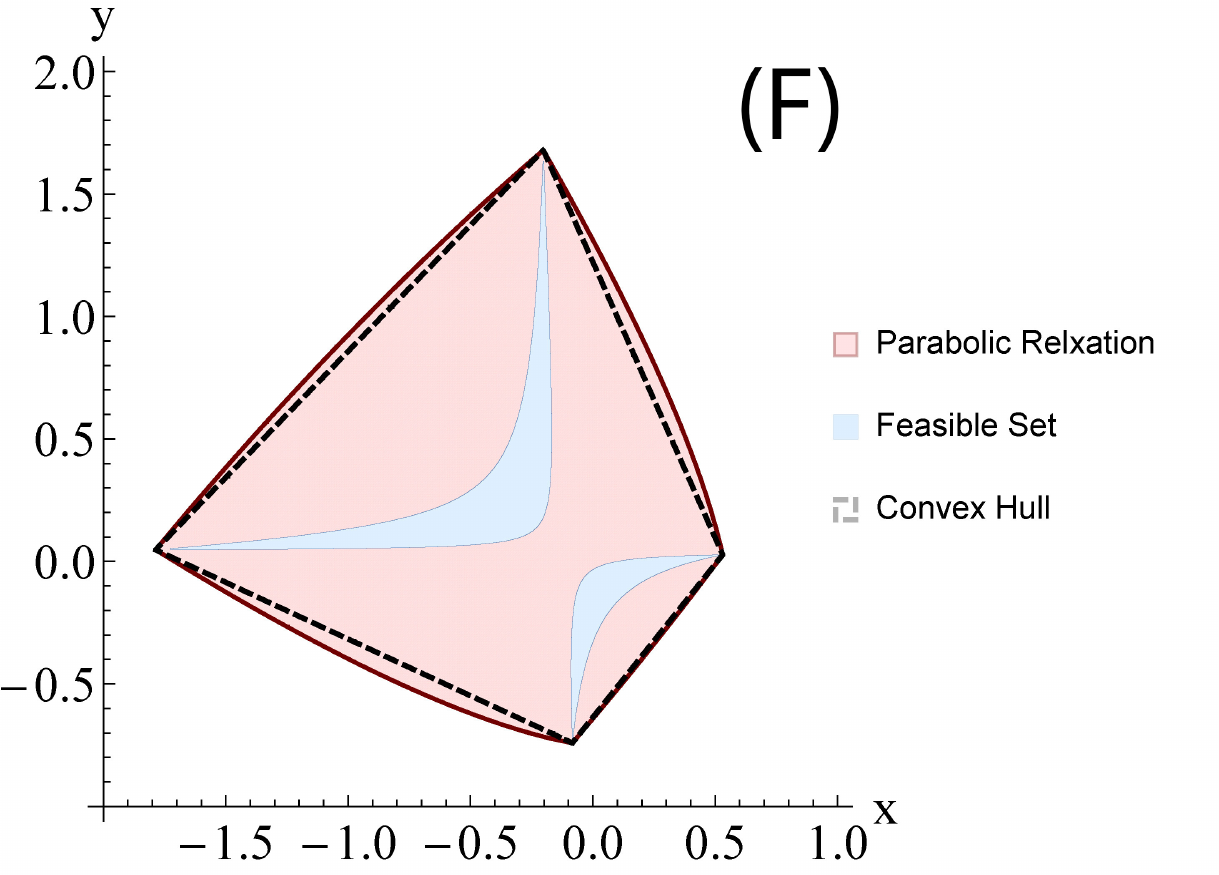}
	\caption{
		Parabolic relaxation and other inner/outer-approximations of SDP cone on the two-dimensional algebraic set defined in \eqref{eq_exmp1}. The original feasible set $\Fcal$ and its convex hull $\mathrm{conv}\{\Fcal\}$ are shown in (A).  The other figures, respectively, demonstrate the regions projected by: 
		(B) SOCP relaxation $\Fcal^{\mathrm{lifted}}\cap\bar{\Scal}^{(2)}_{2,1}$, 
		(C) DD programming $\Fcal^{\mathrm{lifted}}\cap\Scal^{\mathrm{DD}}$, 
		(D) SDD programming $\Fcal^{\mathrm{lifted}}\cap\Scal^{\mathrm{SDD}}$, 
		(E) SDP relaxation $\Fcal^{\mathrm{lifted}}\cap\bar{\Scal}^{(3)}_{2,1}$, and 
		(F) parabolic relaxation $\Fcal^{\mathrm{lifted}}\cap\Pcal_{2,1}$.}
	\vspace{-5mm} 
	\label{fig_para}
\end{figure*}

\subsection{Example 2}
This example demonstrates the application of the sequential penalized parabolic relaxations on the following simple QCQP:
\begin{subequations}
	\begin{align}
		&\hspace{-0.25cm} \underset{\xbf\in\Rbb^5}{\text{minimize}}
		& & x_2^2+7x_5^2-x_2x_3-x_4x_5+2.5x_1+8x_3-x_4  \label{toy_a}\\
		&\hspace{-0.25cm} \text{subject to}
		& & x_1^2-x_3^2+x_1x_4-x_2x_3-x_2x_4-x_4+x_5-1\leq 0, \label{toy_b} \\
		&\hspace{-0.25cm} & & x_1^2+x_2^2+x_3^2+0.5x_2x_5-1.5( x_2+x_5)=0,  \label{toy_c} \\
		&\hspace{-0.25cm} & & x_k\in\{0,1\},\hspace{2.5cm} \forall k\in\{4, 5\}, \label{toy_d}
	\end{align}
\end{subequations}
with the lifted reformulation:
\begin{subequations}
	\begin{align}
		&\hspace{-0.12cm}\underset{\begin{subarray}{l}
				\xbf\in\Rbb^5,~
				\Xbf\in\Sbb_5\end{subarray}}{\text{minimize}}
		& & X_{22}+ 7X_{55}- X_{23}- X_{45}+ 2.5 x_1+ 8x_3- x_4  \label{toy_lifted_a}\\
		&\hspace{-0.12cm} ~\text{subject to}
		& & X_{11}- X_{33}+X_{14}-X_{23}-X_{24}-x_4+x_5-1\leq 0,  \label{toy_lifted_b} \\
		&\hspace{-0.12cm} & &  X_{11}+X_{22}+X_{33}+0.5X_{25}-1.5( x_2+x_5)=0,    \label{toy_lifted_c} \\
		&\hspace{-2.12cm} & &\hspace{-2.12cm}  X_{44}=x_4,\;\;\;\; X_{55}=x_5,\;\;\;\;\max\{0,x_4+x_5-1\}\leq X_{45}\leq \min\{x_4,x_5\}, \label{toy_lifted_d}\\
		&\hspace{-0.12cm} & &  \Xbf = \xbf\xbf\T\!, \label{toy_lifted_e}
	\end{align}
\end{subequations}
in which \eqref{toy_lifted_d} represents integrality constraints and McCormick envelopes. It can be easily verified that $\accentset{\ast}{\xbf}=[-0.2330$, $0.5778$, $-0.6918$, $1,0]^{\!\top}$ is the unique optimal solution of the problem with the objective value $-6.3832$. The relaxations of \eqref{toy_lifted_d} to $(\xbf,\Xbf)\in\Scal^6_{5,1}$ and $(\xbf,\Xbf)\in\Pcal_{5,1}$ (i.e., SDP and parabolic) are both inexact resulting in the objective values $-6.4386$ and $-6.5823$, respectively.

\vspace{3mm}
\paragraph{Mixed-binary search.}
Problem \eqref{toy_a} -- \eqref{toy_d} is a mixed-binary optimization with non-convex quadratic constraints. Using parabolic relaxation, we transform this problem to a mixed-binary convex QCQP which can be fed to these solvers. To this end, one can substitute the constraint \eqref{toy_lifted_e} with
\begin{align}
	(\xbf,\Xbf)\in\{(\xbf,\Xbf)\in\Pcal_{5,1}\;|\;(x_4,x_5)=\{0,1\}^2\},
\end{align}
resulting in a model that is compatible with GUROBI, CPLEX and MOSEK. Interestingly, the relaxation is exact at all four leaves of the binary search. In other words, for every $(i,j)\in\{0,1\}^2$ the substitution of \eqref{toy_lifted_e} with
\begin{align}
	(\xbf,\Xbf)\in\{(\xbf,\Xbf)\in\Pcal_{5,1}\;|\;x_4=i\,\wedge\,x_5=j\},
\end{align}
results in a unique solution $(\xst,\Xst)$ satisfying $\Xst=\xst\xst\T$.
The choices $(i,j)=(0,0),(0,1),(1,0)$, and $(1,1)$ lead to objective values
$-5.3832$,
$-3.5175$,
$-6.3832$ and
$-5.5175$, respectively. 







\vspace{3mm}
\paragraph{Penalization.}
\begin{figure}
	\centering
	\includegraphics[width=1\linewidth]{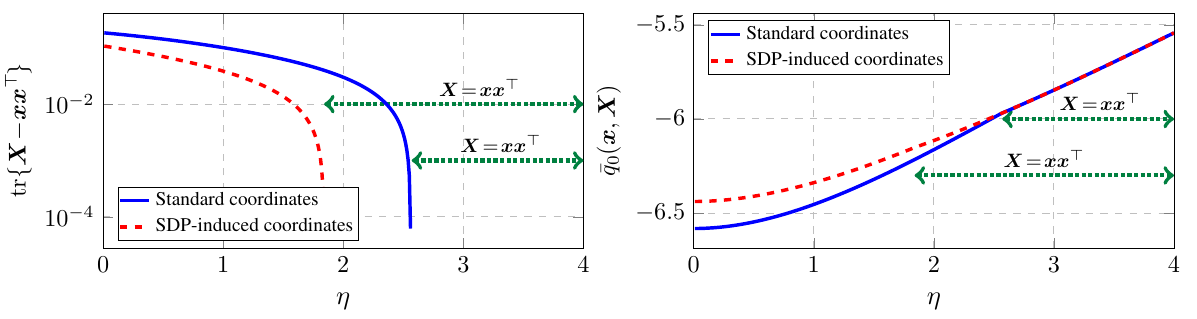}
	\caption{
		Sensitivity of penalization to the choice of penalty parameter $\eta$. For small choices of $\eta$, penalization fails to provide feasible solutions. The feasibility violation $\tr\{\Xbf-\xbf\xbf^{\!\top}\}$ quickly vanishes for the standard and SDP-induced coordinates at $\eta> 2.57$ and $\eta>1.84$, respectively.}
	\label{fig:EtaEffect}
\end{figure}
In order to obtain feasible points from inexact parabolic relaxation, 
we run the Algorithm \ref{al:alg_1}, for ${\begin{bmatrix}0,0,0,0.5,0.5\end{bmatrix}}^{\!\top}$.
The resulting objective values and feasibility violations after the first round for a wide range of $\eta$ values are illustrated in Figure \ref{fig:EtaEffect}. For $\eta=3$, penalization produces a feasible point with the objective value -5.8466.

\begin{table}[t!]
	\centering
	\caption{ Performance of Algorithm \ref{alg:1} on \eqref{toy_a} -- \eqref{toy_d} with starting point 
		$\check{\xbf}=[0,0,0,0.5,0.5]^{\!\top}$, $\eta=2$ and standard coordinates. }
	\scalebox{0.8}{
		\begin{tabular}{|@{\,}c@{\,}| @{\,}c@{\,}| @{\,}c@{\,}| @{\,}c@{\,}|}
			\hline
			 {$\,t\,$} &  {${\xbf}$} &  {$\tr\{ {\Xbf}\!-\!{\xbf}{\xbf}^{\!\top^{\phantom{T}^{\phantom{T}}}}_{\phantom{T}_{\phantom{T}}}\!\!\!\!\!\!\! \}$} &  {$\bar{q}_{\sm{0}{3}}({\xbf}, {\Xbf})$} \\
			\hline
			\hline
			 {0}  &  {${\begin{bmatrix} \phantom{-}0.0000 &    0.0000 &    \phantom{-}0.0000 &    0.5000 &    0.5000\end{bmatrix}}^{\!\top^{\phantom{T}^{\phantom{T}}}}_{\phantom{T}_{\phantom{T}}}\!\!\!\!\!\!$} &  {$--$}   &  {$--$}  \\ \hline
			 {1}  &  {${\begin{bmatrix} -0.2128 &    0.3711 &    -0.6494 &    1.0000 &    0.0309\end{bmatrix}}^{\!\top^{\phantom{T}^{\phantom{T}}}}_{\phantom{T}_{\phantom{T}}}\!\!\!\!\!\!$} &  {$0.0300$}   &  {$-6.1633$}  \\ \hline 
			 {2}  &  {${\begin{bmatrix} -0.2291 &    0.5256 &    -0.6886 &    1.0000 &    0.0098\end{bmatrix}}^{\!\top^{\phantom{T}^{\phantom{T}}}}_{\phantom{T}_{\phantom{T}}}\!\!\!\!\!\!$} &  {$0.0097$}   &  {$-6.3847$}  \\ \hline 
			 {3}  &  {${\begin{bmatrix} -0.2312 &    0.5660 &    -0.6894 &    1.0000 &    0.0000\end{bmatrix}}^{\!\top^{\phantom{T}^{\phantom{T}}}}_{\phantom{T}_{\phantom{T}}}\!\!\!\!\!\!$} &  {$0.0000$}   &  {$-6.3822$}  \\ \hline 
			 {4}  &  {${\begin{bmatrix} -0.2325 &    0.5750 &    -0.6913 &    1.0000 &    0.0000\end{bmatrix}}^{\!\top^{\phantom{T}^{\phantom{T}}}}_{\phantom{T}_{\phantom{T}}}\!\!\!\!\!\!$} &  {$0.0000$}   &  {$-6.3831$}  \\ \hline 
			 {5}  &  {${\begin{bmatrix} -0.2329 &    0.5772 &    -0.6917 &    1.0000 &    0.0000\end{bmatrix}}^{\!\top^{\phantom{T}^{\phantom{T}}}}_{\phantom{T}_{\phantom{T}}}\!\!\!\!\!\!$} &  {$0.0000$}   &  {$-6.3832$}  \\ \hline 
			 {6}  &  {${\begin{bmatrix} -0.2330 &    0.5777 &    -0.6918 &    1.0000 &    0.0000\end{bmatrix}}^{\!\top^{\phantom{T}^{\phantom{T}}}}_{\phantom{T}_{\phantom{T}}}\!\!\!\!\!\!$} &  {$0.0000$}   &  {$-6.3832$}  \\ \hline 
			 {7}  &  {${\begin{bmatrix} -0.2330 &    0.5778 &    -0.6918 &    1.0000 &    0.0000\end{bmatrix}}^{\!\top^{\phantom{T}^{\phantom{T}}}}_{\phantom{T}_{\phantom{T}}}\!\!\!\!\!\!$} &  {$0.0000$}   &  {$-6.3832$}  \\ \hline 
	\end{tabular}}
	\label{tab:Seq_Meth}
\end{table}

\vspace{3mm}
\paragraph{Sequential penalization.}
As shown in Table \ref{tab:Seq_Meth}, higher quality feasible points can be obtained using Algorithm \ref{alg:1} with $\eta=2$. Feasibility is attained after 3 rounds. Distance from optimality is negligible after 7 rounds.

\begin{table}[t!]
	\centering
	\caption{  Performance of Algorithm \ref{alg:1} on \eqref{toy_a} -- \eqref{toy_d} with starting point 
		$\check{\xbf}=[0,0,0,0.5,0.5]^{\!\top}$, $\eta=1$ and SDP-induced coordinates.}
	\scalebox{0.8}{
		\begin{tabular}{|@{\,}c@{\,}| @{\,}c@{\,}| @{\,}c@{\,}| @{\,}c@{\,}|}
			\hline
			 {$\,t\,$} &  {${\xbf}$} &  {$\tr\{\msh{\Xbf}\!-\!{\xbf}{\xbf}^{\!\top^{\phantom{T}^{\phantom{T}}}}_{\phantom{T}_{\phantom{T}}}\!\!\!\!\!\!\!\msh\}$} &  {$\bar{q}_{\sm{0}{3}}({\xbf},\msh{\Xbf})$} \\
			\hline
			\hline
			 {0}  &  {${\begin{bmatrix} \phantom{-}0.0000 &   0.0000 &   \phantom{-}0.0000 &   0.5000 &   0.5000\end{bmatrix}}^{\!\top^{\phantom{T}^{\phantom{T}}}}_{\phantom{T}_{\phantom{T}}}\!\!\!\!\!\!$} &  {$--$}   &  {$--$}  \\ \hline
			 {1}  &  {${\begin{bmatrix}  -0.2308 &   0.4533 &   -0.6968 &   1.0000 &   0.0495\end{bmatrix}}^{\!\top^{\phantom{T}^{\phantom{T}}}}_{\phantom{T}_{\phantom{T}}}\!\!\!\!\!\!$} &  {$0.0384$}   &  {$-6.3385$}  \\ \hline 
			 {2}  &  {${\begin{bmatrix}  -0.2352 &   0.5550 &   -0.7020 &   1.0000 &   0.0190\end{bmatrix}}^{\!\top^{\phantom{T}^{\phantom{T}}}}_{\phantom{T}_{\phantom{T}}}\!\!\!\!\!\!$} &  {$0.0152$}   &  {$-6.3950$}  \\ \hline 
			 {3}  &  {${\begin{bmatrix}  -0.2330 &   0.5745 &   -0.6926 &   1.0000 &   0.0018\end{bmatrix}}^{\!\top^{\phantom{T}^{\phantom{T}}}}_{\phantom{T}_{\phantom{T}}}\!\!\!\!\!\!$} &  {$0.0015$}   &  {$-6.3845$}  \\ \hline 
			 {4}  &  {${\begin{bmatrix}  -0.2329 &   0.5774 &   -0.6917 &   1.0000 &   0.0000\end{bmatrix}}^{\!\top^{\phantom{T}^{\phantom{T}}}}_{\phantom{T}_{\phantom{T}}}\!\!\!\!\!\!$} &  {$0.0000$}   &  {$-6.3832$}  \\ \hline 
			 {5}  &  {${\begin{bmatrix}  -0.2330 &   0.5778 &   -0.6918 &   1.0000 &   0.0000\end{bmatrix}}^{\!\top^{\phantom{T}^{\phantom{T}}}}_{\phantom{T}_{\phantom{T}}}\!\!\!\!\!\!$} &  {$0.0000$}   &  {$-6.3832$}  \\ \hline 
	\end{tabular}}
	\label{tab:SeqI_Ind_Meth}
\end{table}

\vspace{3mm}
\paragraph{SDP-induced coordinates.}
In order to examine the effectiveness of Theorem \ref{thm_cord}, we calculate SDP-induced basis: 
{ \small
\begin{align}
	\wht_1 \! = \! \begin{bmatrix}
		\phantom{-} 0.0000\\ 
		-0.2036\\
		-0.0217\\
		-0.0002\\
		\phantom{-} 0.9788
	\end{bmatrix}\!,
	\wht_2 \! = \! \begin{bmatrix}
		\phantom{-} 0.0000\\
		\phantom{-} 0.0000\\
		\phantom{-} 0.0000\\
		-1.0000 \\
		\phantom{-} 0.0002 
	\end{bmatrix}\!,
	\wht_3 \! = \! \begin{bmatrix}
		\phantom{-} 0.0000 \\
		-0.3170    \\        
		-0.9444            \\
		\phantom{-} 0.0000\\
		-0.0869            
	\end{bmatrix}\!,
	\wht_4 \! = \! \begin{bmatrix}
		-1.0000     \\
		\phantom{-}0.0000               \\
		\phantom{-}0.0000     \\
		\phantom{-}0.0000     \\
		\phantom{-}0.0000             
	\end{bmatrix}\!,
	\wht_5 \! = \! \begin{bmatrix}
		\phantom{-}0.0000  \\
		-0.9263            \\
		\phantom{-}0.3280  \\
		\phantom{-}0.0000  \\
		-0.1855            
	\end{bmatrix}, \nonumber
\end{align}
}
\noindent
via eigen-decomposition of the matrix $\Abf_0+\sum_{k\in\Ecal\cup\Ical}{\hat{\tau}_k\Abf_k}+\sum_{k\in\Vcal}{\hat{\nu}_k\Fbf_{\!k}}$, where $(\hat{\taubf},\hat{\nubf})\in\Rbb^{|\Ecal\cup\Ical|}\!\times\!\Rbb^{|\Vcal|}$ is dual optimal.

As demonstrated in Figure \ref{fig:EtaEffect}, the change of basis allows us to use smaller $\eta$ to achieve feasibility. If $\eta=2$, the objective value is $-6.1146$ which is $4.6\%$ better than penalization using standard basis and $\eta=3$. Additionally, as shown in Table \ref{tab:SeqI_Ind_Meth}, SDP-induced basis allow  Algorithm \ref{alg:1} to arrive to optimality in $5$ rounds using $\eta=1$.

\begin{figure}[t!]
	\centering
	\includegraphics[width=0.35\linewidth]{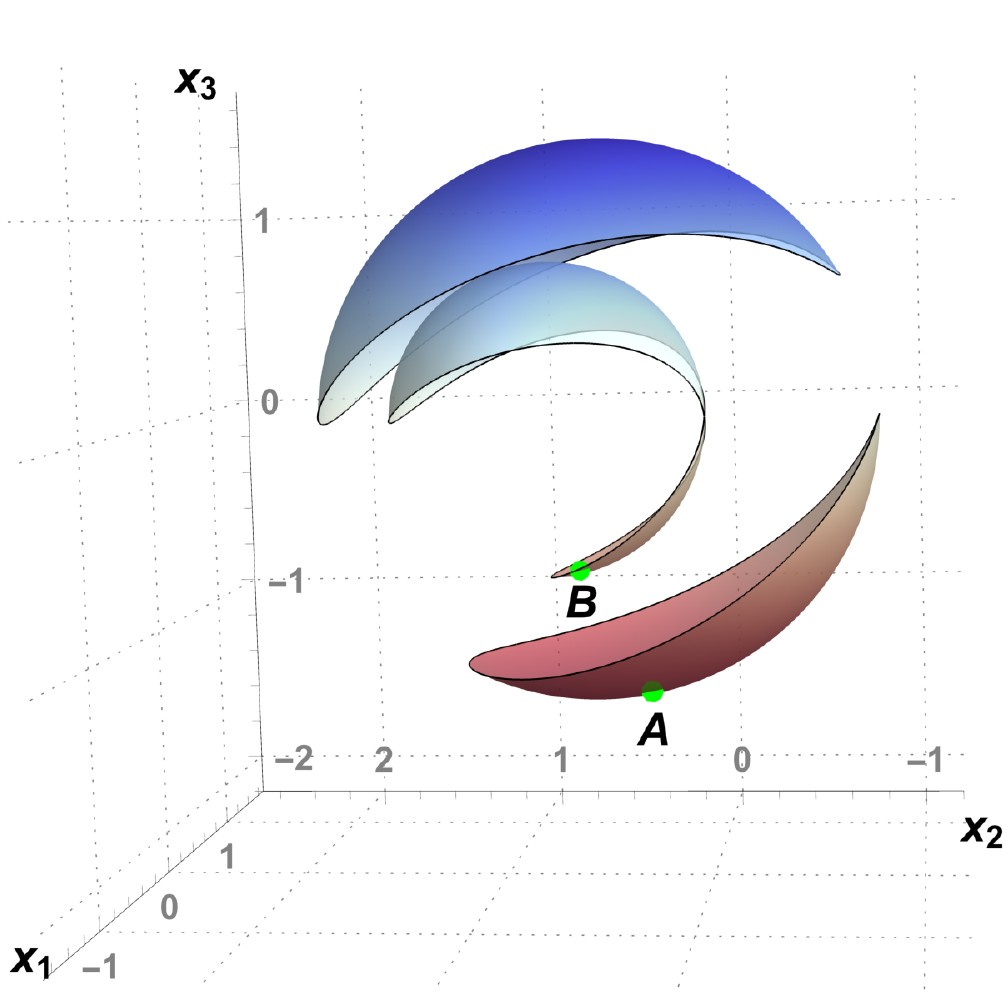} 
	\caption{Feasible set of  problem \eqref{toy_a} -- \eqref{toy_d}. The starting point A and the optimal point, B belong to different connected components of the feasible set.}
	\label{fig:DiscFeas}
\end{figure}

\begin{table}[t!] \vskip -1mm
	\centering
	\caption{Performance of Algorithm \ref{alg:1} for \eqref{toy_a}--\eqref{toy_d} with $\eta=2$ and standard basis.}
	\scalebox{0.8}{
		\begin{tabular}{|@{\,}c@{\,}| @{\,}c@{\,}| @{\,}c@{\,}| @{\,}c@{\,}|}
			\hline
			 {$\,t\,$} &  {${\xbf}$} &  {$\tr\{\msh{\Xbf}\!-\!{\xbf}{\xbf}^{\!\top^{\phantom{T}^{\phantom{T}}}}_{\phantom{T}_{\phantom{T}}}\!\!\!\!\!\!\!\}$} &  {$\bar{q}_{\sm{0}{3}}({\xbf},\msh{\Xbf})$} \\
			\hline
			\hline
			 {0}  &  {${\begin{bmatrix}  -0.3968 &   0.2310 &   -1.2330 &   0.0000 &   1.0000\end{bmatrix}}^{\!\top^{\phantom{T}^{\phantom{T}}}}_{\phantom{T}_{\phantom{T}}}\!\!\!\!\!\!$} &  {$--$}   &  {$--$}  \\ \hline
			 {1}  &  {${\begin{bmatrix}  -0.3462 &   0.3273 &   -1.0676 &   0.3420 &   0.6108 \end{bmatrix}}^{\!\top^{\phantom{T}^{\phantom{T}}}}_{\phantom{T}_{\phantom{T}}}\!\!\!\!\!\!$} &  {$0.4662$}   &  {$-5.3805$}  \\ \hline 
			 {2}  &  {${\begin{bmatrix}  -0.3167 &   0.3987 &   -0.9677 &   1.0000 &   0.4129 \end{bmatrix}}^{\!\top^{\phantom{T}^{\phantom{T}}}}_{\phantom{T}_{\phantom{T}}}\!\!\!\!\!\!$} &  {$0.2424$}   &  {$-6.5111$}  \\ \hline 
			 {3}  &  {${\begin{bmatrix}  -0.2875 &   0.4627 &   -0.8707 &   1.0000 &   0.2476 \end{bmatrix}}^{\!\top^{\phantom{T}^{\phantom{T}}}}_{\phantom{T}_{\phantom{T}}}\!\!\!\!\!\!$} &  {$0.1863$}   &  {$-6.5812$}  \\ \hline 
			 {4}  &  {${\begin{bmatrix}  -0.2641 &   0.5117 &   -0.7939 &   1.0000 &   0.1330 \end{bmatrix}}^{\!\top^{\phantom{T}^{\phantom{T}}}}_{\phantom{T}_{\phantom{T}}}\!\!\!\!\!\!$} &  {$0.1153$}   &  {$-6.5453$}  \\ \hline 
			 {5}  &  {${\begin{bmatrix}  -0.2482 &   0.5443 &   -0.7420 &   1.0000 &   0.0630 \end{bmatrix}}^{\!\top^{\phantom{T}^{\phantom{T}}}}_{\phantom{T}_{\phantom{T}}}\!\!\!\!\!\!$} &  {$0.0590$}   &  {$-6.4790$}  \\ \hline 
			 {6}  &  {${\begin{bmatrix}  -0.2386 &   0.5640 &   -0.7106 &   1.0000 &   0.0233 \end{bmatrix}}^{\!\top^{\phantom{T}^{\phantom{T}}}}_{\phantom{T}_{\phantom{T}}}\!\!\!\!\!\!$} &  {$0.0227$}   &  {$-6.4231$}  \\ \hline 
			 {7}  &  {${\begin{bmatrix}  -0.2331 &   0.5751 &   -0.6928 &   1.0000 &   0.0017 \end{bmatrix}}^{\!\top^{\phantom{T}^{\phantom{T}}}}_{\phantom{T}_{\phantom{T}}}\!\!\!\!\!\!$} &  {$0.0017$}   &  {$-6.3862$}  \\ \hline 
			 {8}  &  {${\begin{bmatrix}  -0.2329 &   0.5773 &   -0.6917 &   1.0000 &   0.0000 \end{bmatrix}}^{\!\top^{\phantom{T}^{\phantom{T}}}}_{\phantom{T}_{\phantom{T}}}\!\!\!\!\!\!$} &  {$0.0000$}   &  {$-6.3832$}  \\ \hline 
			 {9}  &  {${\begin{bmatrix}  -0.2330 &   0.5778 &   -0.6918 &   1.0000 &   0.0000 \end{bmatrix}}^{\!\top^{\phantom{T}^{\phantom{T}}}}_{\phantom{T}_{\phantom{T}}}\!\!\!\!\!\!$} &  {$0.0000$}   &  {$-6.3832$}  \\ \hline 
	\end{tabular}}
	\label{tab:SeqI_Init_Meth}
\end{table}

\vspace{3mm}
\paragraph{Disconnected feasible set.}
Lastly, to assess the sensitivity of Algorithm \ref{alg:1} to the choice of initial point, we employ  $\check{\xbf}=\begin{bmatrix} -0.3968,0.2310,-1.2330,0,1\end{bmatrix}$ which is feasible and locally optimal in the suboptimal connected component $\{\xbf\in\Fcal\;|\;x_4=0\,\wedge\,x_5=1\}$. 
With $\eta=2$ the iterates transition to the correct connected component which contains an optimal point. This procedure is shown in Table \ref{tab:SeqI_Init_Meth} and the initial and terminal points are illustrated in Figure \ref{fig:DiscFeas}. The ability of the algorithm to jump among disconnected components of the feasible region in the original space is particularly advantageous.

\section{Conclusions}
In this paper, we proposed a computationally efficient convex relaxation, named {\it parabolic relaxation,} for non-convex QCQPs. 
Parabolic relaxation includes all feasible points of the QCQP on its boundary, which is a favorable property for recovering feasible points through penalty approaches. While the parabolic relaxation has similar complexity as the common SOCP relaxation, through a basis change, it can be equipped with a bounds as strong as SDP bound. As such, it compares favorably with other relaxations/approximation for QCQP in the literature. In the second in the second part of this work \cite{parabolic_part2}, we give a successive penalized relaxation algorithm and prove its convergence to a KKT point.
The experimental and theoretical results provided show the effectiveness of our proposed approach on benchmark QCQP cases as well as the problem of identifying linear dynamical systems. 

\section{Acknowledgments}
We are grateful to GAMS Development Corporation for providing us with unrestricted access to a full set of solvers throughout the project. 

\bibliographystyle{spbasic}
\bibliography{references}

\end{document}